\documentclass[11pt,reqno,oneside]{amsproc}
\title[Euler immersion]{Euler Immersion}
\author[I.~Kukavica]{Igor Kukavica}
\address{Department of Mathematics, University of Southern California, Los Angeles, CA 90089}
\email{kukavica@usc.edu}
\author[A.~Tuffaha]{Amjad Tuffaha}
\address{Department of Mathematics and Statistics, American University
of Sharjah, Sharjah, UAE}
\email{atufaha\char'100aus.edu}
\author[Q.~Xu]{Qi Xu}
\address{Department of Mathematics, University of Southern California, Los Angeles, CA 90089}
\email{xuqi@usc.edu}
  \chardef\forshowkeys=0
  \chardef\showllabel=0
  \chardef\refcheck=0
  \chardef\sketches=0
\usepackage{enumitem}
\usepackage{datetime}
\usepackage{fancyhdr}
\usepackage{comment}
\allowdisplaybreaks
\ifnum\forshowkeys=1

  \usepackage[notref,notcite,color]{showkeys}
\fi
\usepackage[margin=1in]{geometry}
\usepackage{amsmath, amsthm, amssymb}
\usepackage{times}
\usepackage{graphicx}
\usepackage[dvipsnames,svgnames,table]{xcolor}
\usepackage{marginnote}
\usepackage[unicode,breaklinks=true,colorlinks=true,linkcolor=black,urlcolor=black,citecolor=black]{hyperref}
\usepackage[most]{tcolorbox}
\usepackage{tikz}

\ifnum\refcheck=1
  \usepackage{refcheck}
\fi
\usepackage{tikz}
\usetikzlibrary{arrows.meta,decorations.pathmorphing}
\begin{document}
\def\YY{X}
\def\OO{\mathcal O}
\def\SS{\mathbb S}
\def\CC{\mathbb C}
\def\RR{\mathbb R}
\def\ZZ{\mathbb Z}
\def\HH{\text{H}}
\def\RSZ{\mathcal R}
\def\LL{\mathcal L}
\def\SL{\LL^1}
\def\ZL{\LL^\infty}
\def\GG{\mathcal G}
\def\tt{\langle t\rangle}
\def\erf{\mathrm{Erf}}
\def\mgt#1{\textcolor{magenta}{#1}}
\def\ff{\rho}
\def\gg{G}
\def\sqrtnu{\sqrt{\nu}}
\def\ww{w}
\def\ft#1{#1_\xi}
\def\lec{\lesssim}
\def\gec{\gtrsim}
\def\bard{\bar{\partial}}
\renewcommand*{\Re}{\ensuremath{\mathrm{{\mathbb R}e\,}}}
\renewcommand*{\Im}{\ensuremath{\mathrm{{\mathbb I}m\,}}}
\ifnum\showllabel=1
 \def\llabel#1{\marginnote{\color{lightgray}\rm\small(#1)}[-0.0cm]\notag}
 \def\llabel{\label}
%%%%%%%%%%%%%%%%%%%\def\llabel#1{\label{#1}}
%  \reversemarginpar
%  \def\llabel#1{\notag}
%  \def\llabel#1{\label{#1}}
\else
% % \def\llabel#1{\nonumber}
 \def\llabel#1{\notag}
\fi
\newcommand{\norm}[1]{\left\|#1\right\|}
\newcommand{\nnorm}[1]{\lVert #1\rVert}
\newcommand{\abs}[1]{\left|#1\right|}
\newcommand{\NORM}[1]{|\!|\!| #1|\!|\!|}
\newtheorem{theorem}{Theorem}[section]
\newtheorem{Theorem}{Theorem}[section]
\newtheorem{corollary}[theorem]{Corollary}
\newtheorem{Corollary}[theorem]{Corollary}
\newtheorem{proposition}[theorem]{Proposition}
\newtheorem{Proposition}[theorem]{Proposition}
\newtheorem{Lemma}[theorem]{Lemma}
\newtheorem{lemma}[theorem]{Lemma}
\theoremstyle{definition}
\newtheorem{definition}{Definition}[section]
\newtheorem{Remark}[Theorem]{Remark}
\def\theequation{\thesection.\arabic{equation}}
\numberwithin{equation}{section}
\definecolor{mygray}{rgb}{.6,.6,.6}
\definecolor{myblue}{rgb}{9, 0, 1}
\definecolor{colorforkeys}{rgb}{1.0,0.0,0.0}
\newlength\mytemplen
\newsavebox\mytempbox
\makeatletter
\newcommand\mybluebox{%
    \@ifnextchar[%]
       {\@mybluebox}%
       {\@mybluebox[0pt]}}
\def\@mybluebox[#1]{%
    \@ifnextchar[%]
       {\@@mybluebox[#1]}%
       {\@@mybluebox[#1][0pt]}}
\def\@@mybluebox[#1][#2]#3{
    \sbox\mytempbox{#3}%
    \mytemplen\ht\mytempbox
    \advance\mytemplen #1\relax
    \ht\mytempbox\mytemplen
    \mytemplen\dp\mytempbox
    \advance\mytemplen #2\relax
    \dp\mytempbox\mytemplen
    \colorbox{myblue}{\hspace{1em}\usebox{\mytempbox}\hspace{1em}}}
\makeatother
%Igor's macros  varmac
\def\PPs{\mathcal{P}_{\text{s}}}
\def\JJ{J}
\def\TT{\tilde{T}}
\def\bold{\colu {\bf OLD:}}
\def\eold{\colb {}}
\def\Omegaatm{\Omega_{\text{atm}}}
\def\Omegaoce{\Omega_{\text{oce}}}
\def\Omfone{\Omega_{\text{f1}}}
\def\Omftwo{\Omega_{\text{f2}}}
\def\Omf{\Omega_{\text{f}}}
\def\Ome{\Omega_{\text{e}}}
\def\Gac{\Gamma_{\text{c}}}
\def\Gaf{\Gamma_{\text{f}}}
\def\Gat{\Gamma_{\text{t}}}
\def\Gab{\Gamma_{\text{b}}}

\def\uatm{u_{\text{atm}}}
\def\vatm{v_{\text{atm}}}
\def\watm{w_{\text{atm}}}
\def\patm{p_{\text{atm}}}
\def\uoce{u_{\text{oce}}}
\def\voce{v_{\text{oce}}}
\def\vatm{v_{\text{atm}}}
\def\woce{w_{\text{oce}}}
\def\watm{w_{\text{atm}}}
\def\poce{p_{\text{oce}}}
\def\patm{p_{\text{atm}}}

\def\nablah{\nabla_{\text{H}}}
\def\nablaa{\nabla_{a}}
\def\nablaah{\nabla_{a,\text{H}}}
\def\diva{\div_{a}}
\def\divah{\div_{a,\text{H}}}
\def\divh{\div_{\text{H}}}
\def\rr{r}
\def\weaks{\text{\,\,\,\,\,\,weakly-* in }}
\def\inn{\text{\,\,\,\,\,\,in }}
\def\cof{\mathop{\rm cof\,}\nolimits}
\def\Dn{\frac{\partial}{\partial N}}
\def\Dnn#1{\frac{\partial #1}{\partial N}}
\def\tdb{\tilde{b}}
\def\tda{b}
\def\qqq{u}
\def\lat{\Delta_2}
\def\biglinem{\vskip0.5truecm\par==========================\par\vskip0.5truecm}
\def\inon#1{\hbox{\ \ \ \ \ \ \ }\hbox{#1}}                %in or on
\def\onon#1{\inon{on~$#1$}}
\def\inin#1{\inon{in~$#1$}}
\def\FF{F}
\def\andand{\text{\indeq and\indeq}}
\def\ww{w(y)}
\def\ll{{\color{red}\ell}}
\def\ee{\epsilon_0}   
\def\nnewpage{ }
\def\sgn{\mathop{\rm sgn\,}\nolimits}    
\def\Tr{\mathop{\rm Tr}\nolimits}    
\def\div{\mathop{\rm div}\nolimits}
\def\curl{\mathop{\rm curl}\nolimits}
\def\dist{\mathop{\rm dist}\nolimits}  
\def\supp{\mathop{\rm supp}\nolimits}
\def\indeq{\quad{}}           
\def\period{.}                       
\def\semicolon{\,;}                  
\def\nts#1{{\cor #1\cob}}
\def\colr{\color{red}}
\def\colrr{\color{black}}
\def\colb{\color{black}}
\def\coly{\color{lightgray}}
\definecolor{colorgggg}{rgb}{0.1,0.5,0.3}
\definecolor{colorllll}{rgb}{0.0,0.7,0.0}
\definecolor{colorhhhh}{rgb}{0.3,0.75,0.4}
\definecolor{colorpppp}{rgb}{0.7,0.0,0.2}
\definecolor{coloroooo}{rgb}{0.45,0.0,0.0}
\definecolor{colorqqqq}{rgb}{0.1,0.7,0}
\def\colg{\color{colorgggg}}
\def\collg{\color{colorllll}}
\def\cole{\color{black}}
\def\coleo{\color{colorpppp}}
\def\colu{\color{blue}}
\def\colc{\color{colorhhhh}}
\def\colW{\colb}   %color for weight
\definecolor{coloraaaa}{rgb}{0.6,0.6,0.6}%%%out
\def\colw{\color{coloraaaa}}
\def\comma{ {\rm ,\quad{}} }            
\def\commaone{ {\rm ,\quad{}} }          
\def\nts#1{{\color{blue}\hbox{\bf ~#1~}}} 
\def\ntsf#1{\footnote{\color{colorgggg}\hbox{#1}}} 
\def\blackdot{{\color{red}{\hskip-.0truecm\rule[-1mm]{4mm}{4mm}\hskip.2truecm}}\hskip-.3truecm}
\def\bluedot{{\color{blue}{\hskip-.0truecm\rule[-1mm]{4mm}{4mm}\hskip.2truecm}}\hskip-.3truecm}
\def\purpledot{{\color{colorpppp}{\hskip-.0truecm\rule[-1mm]{4mm}{4mm}\hskip.2truecm}}\hskip-.3truecm}
\def\greendot{{\color{colorgggg}{\hskip-.0truecm\rule[-1mm]{4mm}{4mm}\hskip.2truecm}}\hskip-.3truecm}
\def\cyandot{{\color{cyan}{\hskip-.0truecm\rule[-1mm]{4mm}{4mm}\hskip.2truecm}}\hskip-.3truecm}
\def\reddot{{\color{red}{\hskip-.0truecm\rule[-1mm]{4mm}{4mm}\hskip.2truecm}}\hskip-.3truecm}
\def\tdot{{\color{green}{\hskip-.0truecm\rule[-.5mm]{3mm}{3mm}\hskip.2truecm}}\hskip-.1truecm}
\def\vert{\Vert}
\def\gdot{\greendot}
\def\bdot{\bluedot}
\def\pdot{\purpledot}
\def\ydot{\cyandot}
\def\rdot{\cyandot}
\def\fractext#1#2{{#1}/{#2}}
\def\ii{\hat\imath}
\def\fei#1{\textcolor{blue}{#1}}
\def\vlad#1{\textcolor{cyan}{#1}}
\def\igor#1{\text{ \bf {\textcolor{colorqqqq}IK:{#1}  }   }}
\def\igorf#1{\footnote{\text{{\textcolor{colorqqqq}{#1}}}}}
\newcommand{\p}{\partial}
\newcommand{\UE}{U^{\rm E}}
\newcommand{\PE}{P^{\rm E}}
\newcommand{\KP}{K_{\rm P}}
\newcommand{\uNS}{u^{\rm NS}}
\newcommand{\vNS}{v^{\rm NS}}
\newcommand{\pNS}{p^{\rm NS}}
\newcommand{\omegaNS}{\omega^{\rm NS}}
\newcommand{\uE}{u^{\rm E}}
\newcommand{\vE}{v^{\rm E}}
\newcommand{\pE}{p^{\rm E}}
\newcommand{\omegaE}{\omega^{\rm E}}
\newcommand{\ua}{u_{\rm   a}}
\newcommand{\va}{v_{\rm   a}}
\newcommand{\omegaa}{\omega_{\rm   a}}
\newcommand{\ue}{u_{\rm   e}}
\newcommand{\ve}{v_{\rm   e}}
\newcommand{\omegae}{\omega_{\rm e}}
\newcommand{\omegaeic}{\omega_{{\rm e}0}}
\newcommand{\ueic}{u_{{\rm   e}0}}
\newcommand{\veic}{v_{{\rm   e}0}}
\newcommand{\up}{u^{\rm P}}
\newcommand{\vp}{v^{\rm P}}
\newcommand{\tup}{{\tilde u}^{\rm P}}
\newcommand{\bvp}{{\bar v}^{\rm P}}
\newcommand{\omegap}{\omega^{\rm P}}
\newcommand{\oft}{\Omega_{f_1}}
\newcommand{\ofb}{\Omega_{f_2}}
\newcommand{\tomegap}{\tilde \omega^{\rm P}}
\renewcommand{\up}{u^{\rm P}}
\renewcommand{\vp}{v^{\rm P}}
\renewcommand{\omegap}{\Omega^{\rm P}}
\renewcommand{\tomegap}{\omega^{\rm P}}
\newcommand{\lot}{\text{l.o.t.}}
\colb

\begin{abstract}
We address the Euler immersion problem, a fluid-structure interaction problem in which an elastic body is immersed in an incompressible inviscid fluid governed by the Euler equations. We show that the system exhibits a loss of one derivative and formulate the problem in analytic function spaces. We then prove local well-posedness in analytic spaces under velocity-matching boundary conditions. Finally, by means of an example, we show that existence fails in analytic spaces when both velocity and stress-matching boundary conditions are prescribed.
\colb
\end{abstract}
\colb
\maketitle
\setcounter{tocdepth}{2} 
\tableofcontents

\section{Introduction}\label{sec01}
We consider a model of an elastic structure immersed in an inviscid
incompressible fluid.  The displacement $w$ of the elastic body is modeled
by the wave equation
\begin{equation}
    w_{tt}-\Delta w=0,
    \label{EQ150}
\end{equation}
while the velocity $u$ of the inviscid fluid is described by the Euler
equations
\begin{align}
  \begin{split}
    &
    u_t+(u\cdot\nabla)u+\nabla p=0,
    \\&
    \operatorname{div}u=0
  \end{split}
    \label{EQ2}
\end{align}
in $\Omega_f(t)$, where $\Omega_f(t)$ denotes the fluid domain at time~$t$.  The local existence and stability of the analogous system, with
appropriate conditions, when the Euler equations are replaced by the
Navier-Stokes equations, have been addressed in many references; see, among
many others,
\cite{
ALT,
AT1,
BGT,
BZ,
BoulakiaSchwindtTakahashi2012,
ChambolleDesjardinsEstebanGrandmont2005,
CoutandShkoller2005,CoutandShkoller2006,
DesjardinsEstebanGrandmontLeTallec2001,
GrandmontHillairetLequeurre2019,
IKLT2,
IgnatovaKukavicaLasieckaTuffaha2017,
KO1,
KO2,
Lengeler2014,
MC1}.
In the viscous setting, the parabolic character of the fluid equation provides
smoothing and dissipation which play a central role in the construction of
solutions.  By contrast, the inviscid problem considered here is substantially
more delicate, since the Euler equations do not provide smoothing and the
pressure is determined solely as a Lagrange multiplier enforcing
incompressibility and the boundary constraints.  Related inviscid
free-boundary and Euler-structure models have been studied, for instance, in
\cite{
AydinKukavicaTuffaha2025,
CoutandShkoller2007,
KukavicaTuffaha2024,
Wu1999}.
In particular,
\cite{KukavicaTuffaha2024} establishes local well-posedness for an
Euler-plate free-boundary model, while
\cite{AydinKukavicaTuffaha2025} develops a minimal-regularity theory for the
same class of inviscid flow-structure systems.

As shown in Remark~\ref{R01} below, a formal Sobolev energy approach, using
ALE variables, reveals a derivative-loss mechanism for the
fully coupled inviscid system with both velocity and stress matching.  The
leading pressure boundary term has the correct structure to be paired with
the wave energy.  However, after the integration by parts
in the pressure term
needed to
expose this cancellation, we are left with commutator terms involving the
time-dependent ALE coefficients and the pressure.  In particular, the
commutator associated with
\[
    b_{ki}Dv_i-D(b_{ki}v_i)
\]
forces the top pressure norm~$D\nabla q$.  Under the formal stress relation
$q=\partial_Nw$, the elliptic estimate for the pressure requires one more
boundary derivative of $w$ than is controlled by the Sobolev wave energy at
the same order.  Thus, from the point of view of the direct energy method, the
inviscid problem exhibits a genuine loss of one derivative.

This obstruction suggests that one should work in analytic spaces.  Analytic
and Gevrey regularity for the Euler equations and related inviscid models has
been studied extensively; see, for example,
\cite{
AlinhacMetivier1986,
BardosBenachour1977,
Benachour1979,
KukavicaVicol2009,
KukavicaVicol2011,
LeBail1986}.
A particularly relevant recent work is
\cite{KOS}, where analyticity is used to treat an inviscid inflow-outflow
problem.  The present paper is related to that work in spirit: in both
problems analyticity compensates for a boundary mechanism which is difficult
to handle in a fixed Sobolev scale.  The construction here, however, is
different.  In \cite{KOS}, analyticity is used to handle an inviscid
inflow-outflow formulation, whereas in the present fluid-elastic problem the
analyticity of the wave component determines the interface motion and hence
the ALE coefficients before the Euler pressure is solved.  This leads to a
new compatibility issue which has no analog in the standard fixed-boundary
Euler problem.

Indeed, as it turns out, the analytic approach does not simply recover the
fully coupled system with both velocity and stress matching.  Since the wave
equation is hyperbolic, analytic initial data for $\partial_t w$ can be
extended across the initial position of the interface.  The wave equation can
then be solved, for a short time, from its analytically extended Cauchy data.
This determines the elastic displacement and hence the ALE coefficients for
the fluid.  The pressure is then obtained from the elliptic equation associated
with the incompressibility constraint and the kinematic boundary condition.
In general, this pressure cannot also be prescribed to satisfy an independent
stress matching condition on the interface.  The mismatch example in
Section~\ref{sec07} shows that the pressure produced by the constrained Euler
problem may contain nontrivial boundary modes which cannot be removed by the
usual time-dependent normalization.  Thus, in order to obtain solvability in
the analytic framework, the system is naturally formulated with the velocity
matching condition alone.

The main purpose of this paper is to prove local existence and uniqueness for
analytic initial data for this kinematic Euler-wave coupling.  The analytic
norms used for the fluid variables are denoted by $X_\tau$ and~$Y_\tau$.
Here, $\tau$ is the radius of analyticity, $X_\tau$ is the basic analytic
norm, and $Y_\tau$ is the associated norm which measures one analytic
derivative loss; the precise definitions are given in~\eqref{EQ20}.  The
radius is chosen to decrease in time,
$
    \tau(t)=\tau_0-Mt
$.
When the analytic norm is differentiated in time, this produces the negative
contribution
$
    -M\|v\|_{Y_\tau}
$,
which is the main mechanism used to absorb the derivative-losing terms in the
transport and pressure estimates.

We now describe the main ingredients of the proof.  The elastic component is
first treated on a slightly enlarged strip.  Since $w(0)=0$, the analytic
extension of $\partial_t w(0)$ determines a short-time analytic solution of
the wave equation, and yields uniform bounds for
$
    \partial_t w(t)
$ and
$
    \nabla w(t)$.
The use of an enlarged strip is important: the radius available for the wave
component is chosen strictly larger than the radius used for the fluid
variables.
This condition can be achieved without loss of generality as we can always reduce the radius of analyticity for the initial velocity.
The separation of radii is used to compensate for the
loss caused by differentiating the coupling terms and the ALE coefficients.

The fluid equations are written in arbitrary Lagrangian-Eulerian
(ALE)
coordinates
on a fixed reference domain.  The ALE map $\eta$ is determined by a harmonic
extension of the elastic displacement.  We denote by
  \begin{align}
  \begin{split}
   &a=(\nabla\eta)^{-1},
   \\&
    J=\det\nabla\eta,
   \\&
    b=Ja   
  \end{split}
   \label{EQ114}
  \end{align}
the inverse deformation gradient, the Jacobian, and the cofactor-type matrix,
respectively; see \eqref{EQ09}--\eqref{EQ11} for the precise formulas in the
present geometry.  A key preliminary step is to show that these coefficients
remain close to their Euclidean values for short time.  In particular, we
prove estimates of the form
\[
    \|b-I\|_{X_\tau}+\|a-I\|_{X_\tau}
    \lesssim
    t\,\|\partial_t w(0)\|_{X_{\tau_1}},
\]
together with bounds for $\partial_t\psi$ and for reciprocal quantities
such as~$J^{-1}$.  The Piola identity for $b$ plays a central role in the
pressure equation and in the preservation of the incompressibility constraint.

The pressure estimate is obtained by a perturbative elliptic argument.  Rather
than treating the pressure equation as a fully variable-coefficient elliptic
problem, we rewrite it as
\[
    \Delta q=\partial_j f_j,
\]
where $f_j$ contains the transport terms, the time-dependent ALE
coefficients, and the perturbative error
\[
    (\delta_{jk}-b_{ji}a_{ki})\partial_kq .
\]
The boundary conditions for this elliptic problem are derived from the Euler
equation, the kinematic condition, and the fixed boundary condition.  Since
$I-ba^T$ is small for short time, the term containing $\nabla q$ can be
absorbed.  This perturbative treatment is related in spirit to the classical
estimates for the divergence equation, beginning with Bogovski{\u\i}'s construction
and its subsequent developments; see
\cite{Bogovskii1979,Bogovskii1980,Galdi2011}.  In the present analytic setting
it yields an estimate of the schematic form
\[
    \|\nabla q\|_{X_\tau}
    \lesssim
    (\|\partial_t w(0)\|_{X_{\tau_1}}+1)
    (\|v\|_{X_\tau}+\|v\|_{Y_\tau})(\|v\|_{X_\tau}+1)
    +\|\partial_t w(0)\|_{X_{\tau_1}} .
\]
The occurrence of the $Y_\tau$-norm is the analytic manifestation of the
derivative loss.  Choosing $M$ sufficiently large allows the negative term
$-M\|v\|_{Y_\tau}$ coming from the decreasing radius to absorb the
$Y_\tau$-contributions generated by the transport nonlinearity, the ALE
coefficients, and the pressure estimate.

Finally, the construction is carried out by an iteration in which the next
velocity and the next pressure are solved simultaneously.  This point is
essential.  Both the ALE incompressibility constraint and the kinematic
boundary condition contain the time-dependent coefficients~$b_{ji}$.  If all
terms except $\partial_t v^{(n+1)}$ were evaluated explicitly at the
$n$-th iterate, the constraints would not be preserved at the next step.
We therefore solve $(v^{(n+1)},q^{(n+1)})$ as a coupled linear constrained
system, with the pressure acting as the Lagrange multiplier enforcing the
constraint and the boundary condition.  After proving the solvability of each
linear step, we establish uniform analytic bounds and a contraction estimate
for the differences.  Passing to the limit in the integral formulation of the
iteration gives the desired analytic solution.

\section{The model}
\label{sec02}
We study a free-boundary fluid-structure system consisting of
an elastic body immersed in an inviscid fluid.
For simplicity of presentation, we consider the case when the boundary
is flat. Thus, the initial configuration is as follows.
The initial fluid domain is modeled by the union 
\[
\Omf:=\Omega_1\cup\Omega_2
=\{x:\ x'\in\mathbb{T}^2,\ 1<x_3<2\}\cup \{x:\ x'\in\mathbb{T}^2, -1<x_3<0\}
.
\]
At time $t\geq0$, the upper and lower fluids occupy time-dependent domains
$\Omfone(t)$ and $\Omftwo(t)$, respectively, while the intermediate layer represents the moving elastic domain~$\Ome(t)$. The fluid motion is governed by the incompressible Euler equations, whereas the displacement of the elastic layer is described by a wave equation.
The elastic solid is characterized by its displacement $w$, which solves the linear wave equation 
\begin{align}
    \label{EQ01}
    w_{tt}-\Delta w=0
\end{align}
in~$\Ome\times(0,\infty)$. Moreover, the initial elastic region is
\[
\Ome=\Ome(0)=\{x=(x_1,x_2,x_3): x':=(x_1,x_2)\in\mathbb{T}^2,\ x_3\in(0,1)\}.
\]
The displacement variable $w$ is defined relative to the reference elastic configuration: For each material point in the elastic body, $w$ measures its deviation from the initial position. Consequently, the elastic domain is represented on the fixed reference set $\Ome$, and this set does not depend on time. In particular, one has
\begin{align}
\label{EQ02}
    w|_{t=0}=0.
\end{align}
By contrast, the fluid is described with reference to the domain
\begin{equation}
\Omf:=\Omega_1\cup\Omega_2
=\{x:\ x'\in\mathbb{T}^2,\ 1<x_3<2\}\cup \{x:\ x'\in\mathbb{T}^2,\ -1<x_3<0\}.
   \llabel{EQ03}
     \end{equation}
We write $\Gaf=\Gamma_1\cup\Gamma_2$ for the outer fluid boundary, where
\begin{equation}
\Gamma_1:=\{x:\ x'\in\mathbb{T}^2,\ x_3=2\} \andand
\Gamma_2:=\{x:\ x'\in\mathbb{T}^2,\ x_3=-1\},
   \llabel{EQ04}
     \end{equation}
and we denote by $\Gac=\Gat\cup\Gab$ the interface between the fluid and the elastic body, with
\begin{equation}
\Gat:=\{x:\ x'\in\mathbb{T}^2,\ x_3=1\} \andand
\Gab:=\{x:\ x'\in\mathbb{T}^2,\ x_3=0\}
;
   \llabel{EQ05}
     \end{equation}
see Figure~1 in \cite{KO2} for a sketch.
In the fluid region, the velocity field $u$ satisfies the incompressible Euler equations on a moving domain. Since the fluid domain evolves in time, it is convenient to rewrite the system in the arbitrary Lagrangian-Eulerian (ALE) framework. To this end, we introduce the ALE flow map together with the corresponding displacement function. Since the construction of the ALE coordinates and the corresponding arguments are the same for $\Omega_1$ and $\Omega_2$, up to a reflection and a translation, it suffices to treat the case of~$\Omega_1$.
%The case of $\Omega_2$ can be handled in an entirely analogous manner.

Let $T>0$, to be specified below.
Denote by $\psi\colon\Omega\times[0,T]\to\mathbb{R}$ the harmonic extension of $1+w$ from the top boundary into the reference domain $\Omega=\Omega(0)$, i.e.,
\begin{align}
  \begin{cases}
\Delta \psi =0 & \text{in } \Omega_1,\\
\psi(x_1,x_2,1,t)=1+w(x_1,x_2,t) & \text{on } \Gat\times[0,T],\\
\psi(x_1,x_2,2,t)=2 & \text{on } \Gamma_1\times[0,T].
\end{cases}  
\label{EQ06}
\end{align}
The associated ALE map $\eta\colon\Omega_1\times[0,T]\to\Omega_1(t)$ is defined by
\begin{align}
\eta(x_1,x_2,x_3,t)=\bigl(x_1,x_2,\psi(x_1,x_2,x_3,t)\bigr).
\label{EQ07}
\end{align}
Accordingly,
\begin{align}
\nabla\eta=
\begin{pmatrix}
1&0&0\\
0&1&0\\
\partial_1\psi&\partial_2\psi&\partial_3\psi
\end{pmatrix},
\label{EQ08}
\end{align}
and thus
\begin{align}
a:=(\nabla\eta)^{-1}=\frac1J\,b=
\begin{pmatrix}
1&0&0\\
0&1&0\\
-\partial_1\psi/\partial_3\psi&-\partial_2\psi/\partial_3\psi&1/\partial_3\psi
\end{pmatrix}
\label{EQ09}
\end{align}
with
\begin{align}
J=\partial_3\psi
\label{EQ10}
\end{align}
and
\begin{align}
b=
\begin{pmatrix}
\partial_3\psi&0&0\\
0&\partial_3\psi&0\\
-\partial_1\psi&-\partial_2\psi&1
\end{pmatrix}.
\label{EQ11}
\end{align}
Note that $b$ satisfies the Piola identity
\begin{align}
\partial_i b_{ij}=0
\comma
 j=1,2,3.
\label{EQ12}
\end{align}
Unless explicitly stated otherwise, repeated indices are summed over $1,2,3$. We also introduce the ALE unknowns
\begin{align}
v(x,t)=u(\eta(x,t),t)
\andand
q(x,t)=p(\eta(x,t),t).
\label{EQ13}
\end{align}
In the moving fluid domain, the velocity and pressure satisfy the incompressible Euler equations. Passing to the ALE coordinates introduced above, we rewrite the system on the fixed reference domain $\Omega$ as
\begin{align}
\begin{split}
    &\partial_t v_i
+v_1a_{j1}\partial_j v_i
+v_2a_{j2}\partial_j v_i
+\frac{1}{\partial_3\psi}(v_3-\partial_t\psi)\partial_3 v_i
+a_{ki}\partial_k q=0,
\\
&a_{ki}\partial_k v_i=0
%\quad \text{in } \Omega_1\times[0,T],
\end{split}
\label{EQ14}
\end{align}
in $\Omega_1\times[0,T]$,
where we used the relation
\begin{equation}
a_{j3}\partial_j v_i=\frac{1}{\partial_3\psi}\partial_3 v_i.
   \llabel{EQ15}
     \end{equation}
The initial condition is
\begin{align}
v(\cdot,0)=v_0=(v_{01},v_{02},v_{03}).
\label{EQ16}
\end{align}
On the far-top and far-bottom fluid boundary $\Gamma_1$ and $\Gamma_2$, we impose
\begin{align}
v_3=0.
\label{EQ17}
\end{align}
Finally, assuming the velocity matching condition at the fluid-structure interface, the kinematic boundary condition on $\Gat$ takes the form
\begin{align}
b_{3i}v_i=\partial_tw.
\label{EQ18}
\end{align}
For $\alpha=(\alpha_1,\alpha_2,\alpha_3)\in\mathbb N_0^3$, we set
\begin{equation}
|\alpha|=\alpha_1+\alpha_2+\alpha_3,\qquad
\partial^\alpha=\partial_1^{\alpha_1}\partial_2^{\alpha_2}\partial_3^{\alpha_3},\andand
\epsilon^\alpha=\epsilon_1^{\alpha_1}\epsilon_2^{\alpha_2}\epsilon_3^{\alpha_3};
   \llabel{EQ19}
     \end{equation}
The proof shows that we may simply fix
$\epsilon=(\epsilon_1,\epsilon_2,\epsilon_3)=(1,1,\frac{1}{2})$.
%Since this paper is concerned with analytic function spaces, we introduce the following norm in order to quantify analyticity more precisely,
Introduce the analytic norms
\begin{align}
    \begin{split}
    \label{EQ20}
        &\|u\|_{X_\tau}:=\sum_{m=0}^\infty\sum_{|\alpha|=m}\epsilon^\alpha 
        \tau^{(m-3)_+}M_m\|\partial^\alpha u\|, \quad \text{where }M_m=\frac{(m+1)^3}{m!},
        \\
        &
        \|u\|_{Y_\tau}:=\sum_{m=4}^\infty\sum_{|\alpha|=m}\epsilon^\alpha 
        \tau^{m-4}mM_m\|\partial^\alpha u\|,
    \end{split}
\end{align}
where we used the notation
\begin{equation}
\|\cdot\|:=\|\cdot\|_{L^2(\Omega)}
   \llabel{EQ21} 
     \end{equation}
for the $L^2$~norm
with the domain that should be clear from the context.
Note that functions for which either of the two norms is finite
allow a unique analytic extension on larger domains comparable to the
analyticity radius.
Here, $\tau(t)$ denotes a time-dependent radius of analyticity. The choice of $\tau(t)$ is the key degree of freedom in the analytic framework, and a suitable choice of $\tau(t)$ yields the analytic dissipation. This mechanism allows us to compensate for the loss of derivatives; see Section~\ref{sec05} and the discussion below.

Throughout the paper, $C$ denotes a positive constant whose value may vary from line to line. We write
$A\lec B$
to mean $A \leq C B$ for a constant $C$, and
$A\sim B$
to mean $A \lesssim B$ and $B \lesssim A$.
We first discuss the compatibility condition imposed on the average of the
elastic velocity. Assume that $\partial_tw(\cdot,0)$ admits an analytic
extension to the enlarged strip
\[
\mathbb T^2\times(-\delta,1+\delta),
\]
for some $0<\delta<\frac{1}{2}$.
Since $w(\cdot,0)=0$, the Cauchy problem for the wave equation with initial
data
\[
w(\cdot,0)=0
\andand
\partial_tw(\cdot,0)
\]
has a short-time analytic solution on the enlarged strip. By uniqueness, this
solution agrees with the original elastic displacement on
$\mathbb T^2\times(0,1)$. We shall use this extension without changing the
notation.
Define
\begin{equation}
F_0(x_3):=
\int_{\mathbb T^2}
\partial_tw(x_1,x_2,x_3,0)\,dx_1dx_2
\comma
x_3\in(-\delta,1+\delta).
   \llabel{EQ22}
\end{equation}
For
\[
\overline w(x_3,t):=
\int_{\mathbb T^2}w(x_1,x_2,x_3,t)\,dx_1dx_2,
\]
we have
\[
\partial_{tt}\overline w-\partial_{33}\overline w=0
\comma
\overline w(\cdot,0)=0,
\quad
\text{and}
\quad
\partial_t\overline w(\cdot,0)=F_0 .
\]
Hence, the one-dimensional d'Alembert formula gives
\begin{equation}
\partial_t\overline w(x_3,t)
=
\frac12\bigl(F_0(x_3+t)+F_0(x_3-t)\bigr)
   \llabel{EQ23}
\end{equation}
as long as $x_3\pm t\in(-\delta,1+\delta)$.
Thus the following condition is necessary and sufficient for the averages of
the prescribed elastic normal velocity at the two interfaces to vanish for short
time:
\begin{align}
F_0(-x_3)=-F_0(x_3)\comma 0<x_3<\delta,
\andand
F_0(1+s)=-F_0(1-s)\comma 0<s<\delta .
   \label{EQ24}
\end{align}
In particular,
\begin{equation}
\int_{\Gat}\partial_tw\,dx_1dx_2=0,
   \llabel{EQ25}
\end{equation}
which is the flux compatibility condition associated with
$a_{ki}\partial_kv_i=0$, $v_3=0$ on $\Gamma_1$, and
$b_{3i}v_i=\partial_tw$ on~$\Gat$.
We assume throughout that the initial data satisfy
\begin{align}
    \partial_i v_{0i}=0
    \inin{ \Omega_1}
   \llabel{EQ26}
\end{align}
and
  \begin{align}
  \begin{split}
   &
    v_{03}=0,
    \onon{\Gamma_1},
    \\&
    v_{03}=\partial_tw(\cdot,0)
    \onon{\Gat}
    .
  \end{split}
   \label{EQ27}
  \end{align}
Since $w(\cdot,0)=0$, the ALE map is the identity at $t=0$. Equivalently,
\begin{align}
  \begin{split}
    &
    a_{ki}(\cdot,0)\partial_k v_{0i}=0
    \inin{\Omega_1}
    ,
    \\&
    v_{03}=0
    \inin{\Gamma_1},
    \\&
    b_{3i}(\cdot,0) v_{0i}=\partial_tw(\cdot,0)
  \onon{\Gat}
  \end{split}
   \label{EQ28}
\end{align}
on~$\Gat$.

The following is our main theorem.

\begin{Theorem}[Local well-posedness of \eqref{EQ14} in analytic spaces]
\label{T01}
Assume that $v_0\in X_{\tau_0}\cap Y_{\tau_0}$, for some $\tau_0>0$.
Assume also that $w$ satisfies \eqref{EQ02} and that
\[
\partial_tw(\cdot,0)\in
X_{\tau_1}\bigl(\mathbb T^2\times(-\delta,1+\delta)\bigr)
,
\]
for some $0<\delta<1/2$ and $\tau_1>\tau_0$. Assume that the
compatibility conditions \eqref{EQ24} and \eqref{EQ28} hold.
Choose $M\geq1$ sufficiently large, depending only on
\begin{equation}
A:=\|\partial_tw(\cdot,0)\|_{X_{\tau_1}(\mathbb T^2\times(-\delta,1+\delta))}
   \llabel{EQ112}
\end{equation}
and
\begin{equation}
B:=4\bigl(\|v_0\|_{X_{\tau_0}}+\|v_0\|_{Y_{\tau_0}}+1\bigr)
,
   \llabel{EQ113}
\end{equation}
and set
\begin{equation}
\tau(t):=\tau_0-Mt
.
   \llabel{EQ29}
\end{equation}
Then there exists $T_0>0$, depending only on $A$, $B$, $M$, $\tau_0$,
and $\delta$, such that the system \eqref{EQ14}, with boundary conditions
\eqref{EQ17} and \eqref{EQ18}, has a unique solution $v$ on $[0,T_0]$.
Moreover,
\begin{equation}
v\in C\bigl([0,T_0];X_{\tau(t)}\bigr)
\cap
L^1\bigl([0,T_0];Y_{\tau(t)}\bigr)
   \llabel{EQ30}
\end{equation}
and
\begin{equation}
\sup_{0\leq t\leq T_0}\|v(t)\|_{X_{\tau(t)}}
+
M\int_0^{T_0}\|v(t)\|_{Y_{\tau(t)}}\,dt
\leq B .
   \label{EQ31}
\end{equation}
\end{Theorem}
\begin{Remark}
The estimate \eqref{EQ31}
is the basic a priori estimate for the proof.
Note that 
the
decrease of the analytic radius produces the 
$L_t^1Y_{\tau(t)}$ control needed
to absorb the derivative loss in the nonlinear terms.
\end{Remark}

%\begin{Remark}[Analysis of a possible Sobolev approach]
\begin{Remark}[A formal Sobolev obstruction]\label{R01}
We explain why the analytic framework is natural for the present problem.  The
following discussion is formal and concerns the direct Sobolev energy method
for the fully coupled problem in which we also try to impose the stress
matching condition $q=\partial_N w$ on the fluid-structure interface.  This formal computation is used only to identify the derivative-loss mechanism motivating the analytic framework.  The
purpose of the discussion is to identify the derivative-loss mechanism which
appears in the Sobolev energy estimation.
Suppose that one tries to close a Sobolev energy at order $s$, and let
\[
    D=\bar\partial^m\partial_t^n
   \comma
   \text{where~}
    m+n=s,
\]
with $\bar\partial$ denoting tangential differentiation.  In the differentiated
fluid equation, the pressure contribution contains
\[
    \int_{\Omega_1}D(b_{ki}\partial_kq)\,Dv_i .
\]
This term cannot be estimated directly as such an estimate would require
the top norm $D\nabla q$, while $q$ is a Lagrange multiplier rather than an
energy variable.  Thus one is forced to integrate the principal pressure part
by parts in order to use the ALE incompressibility constraint.
Adding and subtracting the principal part gives
\[
\begin{aligned}
\int_{\Omega_1}D(b_{ki}\partial_kq)\,Dv_i
&=
\int_{\Omega_1}D\partial_kq\,D(b_{ki}v_i)
+\int_{\Omega_1}
\bigl(D(b_{ki}\partial_kq)-b_{ki}D\partial_kq\bigr)Dv_i
\\&\indeq
+\int_{\Omega_1}
\bigl(b_{ki}Dv_i-D(b_{ki}v_i)\bigr)D\partial_kq .
\end{aligned}
\]
Since $\partial_k(b_{ki}v_i)=0$, the first term reduces to a boundary term
\[
    \int_{\Omega_1}D\partial_kq\,D(b_{ki}v_i)
    =
    \int_{\partial\Omega_1}Dq\,D(b_{ki}v_i)N_k .
\]
On the fluid-structure interface this term has the form
\[
    \int_{\Gat}Dq\,D(b_{3i}v_i).
\]
At the level of the formal Sobolev energy for the fully coupled system, one
would use the kinematic condition $b_{3i}v_i=\partial_t w$ together with the
stress matching condition $q=\partial_N w$.  The boundary term then becomes
\[
    \int_{\Gat}D\partial_Nw\,D\partial_t w .
\]
This has the correct structure to be paired with the corresponding boundary
term in the differentiated wave energy, after integrations by parts in time in
the mixed estimates.  Thus the leading pressure boundary contribution is not
the main obstruction.

The difficulty lies in the commutators left after this integration by parts.
The first commutator is
\[
    \int_{\Omega_1}
    \bigl(D(b_{ki}\partial_kq)-b_{ki}D\partial_kq\bigr)Dv_i .
\]
This is of the form $[D,b]\nabla q$, and standard Sobolev commutator estimates
give schematically
\[
    \|[D,b]\nabla q\|_{L^2}
    \lesssim
    \|\nabla b\|_{L^\infty}\|\nabla q\|_{H^{s-1}}
    +
    \|b\|_{H^s}\|\nabla q\|_{L^\infty}.
\]
Thus this term only requires $\nabla q$ one order below the top level, and it
may be compatible with the elliptic gain for the pressure.
The second commutator has a different structure:
\[
    \int_{\Omega_1}
    \bigl(b_{ki}Dv_i-D(b_{ki}v_i)\bigr)D\partial_kq .
\]
Indeed,
\[
    b_{ki}Dv_i-D(b_{ki}v_i)
    =
    -(Db_{ki})v_i+\text{lower-order terms}.
\]
Hence, this commutator contains the top-order contribution
\[
    -\int_{\Omega_1}(Db_{ki})v_i\,D\partial_kq .
\]
A direct H\"older-Sobolev estimate yields
\[
\left|
\int_{\Omega_1}(Db_{ki})v_i\,D\partial_kq
\right|
\lesssim
\|v\|_{L^\infty}\|Db\|_{L^2}\|D\nabla q\|_{L^2}.
\]
Therefore, this term forces the top pressure norm
\[
    \|D\nabla q\|_{L^2}\sim \|\nabla q\|_{H^s}.
\]
In the formal Sobolev setting with the stress condition $q=\partial_N w$ on
$\Gat$, an elliptic estimate for $\|\nabla q\|_{H^s(\Omega_1)}$ involves
the boundary norm
\[
    \|q\|_{H^{s+1/2}(\Gat)}
    =
    \|\partial_Nw\|_{H^{s+1/2}(\Gat)} .
\]
The wave energy at order $s$ controls
\[
    \|\partial_t w\|_{H^s(\Omega_e)}
    +
    \|\nabla w\|_{H^s(\Omega_e)},
\]
and thus also $w$ at the level~$H^{s+1}(\Omega_e)$.  This is
insufficient to control
$\|\partial_Nw\|_{H^{s+1/2}(\Gat)}$, which requires one additional spatial
derivative of~$w$.  Thus the pressure commutator, together with the stress
boundary condition in the elliptic estimate for $q$, leads to a
one-derivative loss in a fixed Sobolev scale.
In the analytic framework, the same loss is measured by the $Y_\tau$-norm.
The pressure estimate has the schematic form
\[
    \|\nabla q\|_{X_\tau}\lec (\|\partial_t w(0)\|_{X_{\tau_1}}+1)
    (\|v\|_{X_\tau}+\|v\|_{Y_\tau})(\|v\|_{X_\tau}+1)+\|\partial_t w(0)\|_{X_{\tau_1}}.
\]
Since the analytic radius is decreasing, $\tau(t)=\tau_0-Mt$, differentiating
the analytic norm produces the negative term
\[
    -M\|v\|_{Y_\tau}.
\]
Choosing $M$ sufficiently large allows this term to absorb the
$Y_\tau$-contribution generated by the pressure and transport estimates.
This is the mechanism used below.

The mismatch example in Section~\ref{sec07} complements this observation.  Once
the analytic wave displacement fixes the ALE coefficients, the Euler pressure
is determined by the incompressibility constraint and the kinematic condition.
Its boundary trace is generally incompatible with an additional stress matching
condition, which is why the analytic problem studied here is formulated with
the kinematic condition alone.
\end{Remark}

We now turn to the proof of Theorem~\ref{T01}.

\section{Preliminary results}
\label{sec03}
In this section, we prove finite-time analyticity for the wave equation and derive elliptic estimates for the ALE coordinates, which will play a key role in the subsequent analysis.

\begin{Lemma}[Short-time analyticity for the wave equation]
\label{L01}
    Assume that $\partial_tw(\cdot,0)\in X_{\tau_1}(\mathbb{T}^2\times(-\delta,1+\delta))$.
Then $\partial_tw(\cdot,t),\nabla w(\cdot,t)\in X_{\tau_1}(\mathbb{T}^2\times(0,1))$ for $0\leq t\leq \delta$.
\end{Lemma}
\begin{proof}[Proof of Lemma~\ref{L01}]
We define
    \begin{align}
    \llabel{EQ32}
        G_\alpha(t)=\int_{\mathbb{T}^2}\int_{-\delta+t}^{1+\delta-t}
    \Bigl(
    (\partial^\alpha\partial_t w)^2+|\partial^\alpha\nabla w|^2\, dx_3
    \Bigr)
    \,dx',
    \end{align}
    where
    \begin{equation}
      |\partial^\alpha\nabla w|^2:=\sum_{i=1}^3(\partial^\alpha\partial_iw)^2
   \llabel{EQ33}
     \end{equation}
for any $\alpha\in\mathbb{N}_0^3$.
A direct differentiation in time yields
    \begin{align}
    \begin{split}
        \label{EQ34}
        G'_\alpha(t)&=2\int_{\mathbb{T}^2}
	\int_{-\delta+t}^{1+\delta-t}\partial^\alpha\partial_t w
        \partial^\alpha\partial_{tt} w
        +\partial^\alpha\nabla w\cdot\partial^\alpha\nabla\partial_t w
         \, dx_3\,dx'
        \\&\indeq
        -\int_{\mathbb{T}^2}
       \Bigl(
        (\partial^\alpha\partial_t w)^2(x',-\delta+t,t)
        +(\partial^\alpha\nabla w)^2(x',-\delta+t,t)\,dx'
       \Bigr)
        \\&\indeq
        -\int_{\mathbb{T}^2}
       \Bigl(
        (\partial^\alpha\partial_t w)^2(x',1+\delta-t,t)
        +(\partial^\alpha\nabla w)^2(x',1+\delta-t,t)\,dx'
	\Bigr)
        .
        \end{split}
    \end{align}
For convenience, we define
    \begin{equation}
        \tilde G_\alpha(x_3,t)
	=\int_{\mathbb{T}^2}
        \Bigl(
	(\partial_t\partial^\alpha w)^2(x',x_3,t)
        +(\partial^\alpha\nabla w)^2(x',x_3,t)
        \Bigr)
       \,dx'
    .
   \llabel{EQ35}
     \end{equation}
Integrating by parts in \eqref{EQ34}, we obtain
    \begin{align}
        \begin{split}
            G'_\alpha(t)&=-2\int_{\mathbb{T}^2}\int_{-\delta+t}^{1+\delta-t}
            \partial^\alpha\partial_t w\partial^\alpha(\partial_{tt}w-\Delta w)\, dx_3\,dx'
            \\&\indeq
            +2\int_{\mathbb{T}^2}\partial^\alpha\partial_t w\partial^\alpha\frac{\partial w}{\partial N}\Bigl|_{x_3=-\delta+t}^{x_3=1+\delta-t}\,dx'
            -\tilde G_\alpha(-\delta+t,t)
            -\tilde G_\alpha(1+\delta-t,t)
            \\&\leq
            -\int_{\mathbb{T}^2}(\partial^\alpha\partial_tw-\partial^\alpha\partial_3 w)^2
            \Big|_{x_3=1+\delta-t}
            \,dx'
            -\int_{\mathbb{T}^2}(\partial^\alpha\partial_tw+\partial^\alpha\partial_3 w)^2
            \Big|_{x_3=-\delta+t}
            \,dx'
            \leq 0
     .
        \end{split}
        \llabel{EQ36}
    \end{align}
Since $w(0)=0$ on $\Ome$,
    \begin{equation}
        G_\alpha(t)\leq G_\alpha(0)= \|\partial^\alpha\partial_tw(0)\|_{L^2(\mathbb{T}^2\times(-\delta,1+\delta))}^2
   \llabel{EQ37}    
     \end{equation}
for any $0\leq t\leq \delta$.
Therefore,
    \begin{align}
      \begin{split}
    &
    \|\partial^\alpha\partial_tw(t)\|_{L^2(\Ome)}
    +\|\partial^\alpha\nabla w(t)\|_{L^2(\Ome)}\leq 2(G_\alpha(t))^\frac{1}{2}\leq2(G_\alpha(0))^\frac{1}{2}
    \\&\indeq
    =2\|\partial^\alpha\partial_tw(0)\|_{L^2(\mathbb{T}^2\times(-\delta,1+\delta))} 
    .
  \end{split}
    \label{EQ38}
    \end{align}
Substituting \eqref{EQ38} into~\eqref{EQ20}, we obtain
    \begin{align}
        \begin{split}
        \llabel{EQ39}
    &\|\partial_tw(t)\|_{X_{\tau_1}(\Ome)}+\|\nabla w(t)\|_{X_{\tau_1}(\Ome)}
    \\&\indeq
    =\sum_{m=0}^\infty\sum_{|\alpha|=m}\epsilon^\alpha 
    \tau_1^{(m-3)_+}M_m(\|\partial^\alpha \partial_t w\|_{L^2(\Ome)}+\|\partial^\alpha \nabla w\|_{L^2(\Ome)})
    \\&\indeq
    \leq
    2\sum_{m=0}^\infty\sum_{|\alpha|=m}\epsilon^\alpha 
    \tau_1^{(m-3)_+}M_m\|\partial^\alpha\partial_tw(0)\|_{L^2(\mathbb{T}^2\times(-\delta,1+\delta))} 
    \leq 2\|\partial_t w(0)\|_{X_{\tau_1}}
        \end{split}
    \end{align}
for $0\leq t\leq \delta$.
\end{proof}

\begin{Lemma}[Short-time analytic estimates for the ALE coefficients]
\label{L02}
Suppose that $\partial_t w(\cdot,0)\in X_{\tau_1}(\mathbb{T}^2\times(-\delta,1+\delta))$. 
Then there exists a sufficiently small time $T>0$ such that the ALE coefficients satisfy the following estimates for all $0\le t\le T$:
\begin{enumerate}
    \item $\|b-I\|_{X_\tau}\lec \|\partial_tw(0)\|_{X_{\tau_1}}t$;
    \item $\|a-I\|_{X_\tau}\lec \|\partial_tw(0)\|_{X_{\tau_1}}t$;
    \item $\|\partial_t\psi\|_{X_\tau}\lec\|\partial_tw(0)\|_{X_{\tau_1}}$.
\end{enumerate}
\end{Lemma}
\begin{proof}[Proof of Lemma~\ref{L02}]
We prove the three assertions separately.

\medskip
\noindent\textbf{Proof of (1).}
    We write $\alpha':=(\alpha_1,\alpha_2)\in \mathbb{N}_0^2$ and $\epsilon'=(\epsilon_1,\epsilon_2)=(1,1)$. We further define
    \begin{equation}
    S_k=\sum_{m=0}^\infty\sum_{|\alpha'|=m}\sum_{i=1}^3
    \|\partial^{\alpha'}\partial_3^k\partial_i\tilde\psi\|\tau^{(m+k-3)_+}(\epsilon')^{\alpha'}\epsilon_3^k M_{m+k},
   \llabel{EQ40}
     \end{equation}
where $\tilde\psi=\psi-x_3$,
for any $k\in \mathbb{N}_0$.
Using~\eqref{EQ06}, we obtain
    \begin{align}
        \begin{split}
            \llabel{EQ41}
            S_k
            &\leq\sum_{m=0}^\infty\sum_{|\alpha'|=m}\sum_{i=1}^3
            (\|\partial^{\alpha'}\partial_1^2\partial_3^{k-2}\partial_i\tilde\psi\|+
            \|\partial^{\alpha'}\partial_2^2\partial_3^{k-2}\partial_i\tilde\psi\|
            )\tau^{(m+k-3)_+}(\epsilon')^{\alpha'}\epsilon_3^k M_{m+k}
            \\&
            \leq \sum_{m=0}^\infty\sum_{|\alpha'|=m}\sum_{i=1}^3\sum_{j=1}^2
            \|\partial^{\alpha'}\partial_j^2\partial_3^{k-2}\partial_i\tilde\psi\|\tau^{(m+k-3)_+}
            (\epsilon')^{\alpha'}\epsilon_j^2\epsilon_3^{k-2}\underbrace{(\epsilon_3\epsilon_j^{-1})^2}_{=1/4}M_{m+k}
            \\&
            \leq \frac{1}{4}S_{k-2}.
        \end{split}
    \end{align}
Therefore, we deduce the elliptic estimate for at most one order normal derivative
    \begin{align}
      \Vert b-I\Vert_{X_{\tau}}
      \lec
        \|\nabla\tilde\psi\|_{X_\tau}\leq 3(S_0+S_1).
        \label{EQ42}
    \end{align}
We will estimate $S_0$ and $S_1$ separately using~\eqref{EQ06}.
Starting with $S_1$, we have
    \begin{align}
        \begin{split}
        \label{EQ43}
    S_1
    &=\sum_{m=0}^\infty\sum_{|\alpha'|=m}\sum_{i=1}^3
    \|\partial^{\alpha'}\partial_3\partial_i\tilde\psi\|(\epsilon')^{\alpha'}\epsilon_3
    \tau^{(m-2)_+}M_{m+1}
    \\&\leq
    \sum_{m=0}^\infty\sum_{|\alpha'|=m}\sum_{i=1}^2
    \|\partial^{\alpha'}\partial_i\partial_3\tilde\psi\|
    (\epsilon')^{\alpha'}\epsilon_i(\epsilon_3\epsilon_i^{-1})\tau^{(m-2)_+}M_{m+1}
    \\&\indeq
    +\sum_{m=0}^\infty\sum_{|\alpha'|=m}\sum_{i=1}^2\|\partial^{\alpha'}\partial_i^2\tilde\psi\|
    (\epsilon')^{\alpha'}\epsilon_i(\epsilon_3\epsilon_i^{-1})\tau^{(m-2)_+}M_{m+1}
    \\&\lec
    S_0+
    \sum_{m=0}^\infty\sum_{|\alpha'|=m}\sum_{i=1}^2\|\partial^{\alpha'}\partial_iw\|
    _{H^\frac{1}{2}(\tau_t)}
    (\epsilon')^{\alpha'}\epsilon_i\tau^{(m-2)_+}M_{m+1}
    .
        \end{split}
    \end{align}
 By the trace theorem, Young's inequality, and the initial condition~\eqref{EQ02}, we then obtain
    \begin{align}
        \begin{split}
            \label{EQ44}
            S_1
	    &\lec
            S_0+\sum_{m=0}^\infty\sum_{|\alpha'|=m}\sum_{i=1}^2
            \|\partial^{\alpha'}\partial_iw\|_{L^2(\Ome)}^\frac{1}{2}
            \|\partial^{\alpha'}\partial_i\nabla w\|_{L^2(\Ome)}^\frac{1}{2}
            (\epsilon')^{\alpha'}\epsilon_i\tau^{(m-2)_+}M_{m+1}
            \\&
	    \lec
            S_0+
            \sum_{m=1}^\infty\sum_{|\alpha|=m}
            \left\|\int_0^t\partial^\alpha \partial_tw\,ds\right\|_{L^2(\Ome)}
            \epsilon^\alpha\tau^{(m-3)_+}M_{m}
            \\&\indeq
            +\sum_{m=1}^\infty\sum_{|\alpha|=m}
            \left\|\int_0^t\partial^\alpha\nabla \partial_t w\,ds\right\|_{L^2(\Ome)}
            \epsilon^\alpha\tau^{(m-3)_+}M_{m}
            \\&\lec
            S_0+
            \int_0^t\left(\sum_{m=1}^\infty\sum_{|\alpha|=m}\left\|\partial^\alpha \partial_tw\right\|_{L^2(\Ome)}
            \epsilon^\alpha\tau^{(m-3)_+}M_{m}\right)\,ds
            \\&\indeq+
            \int_0^t\left(\sum_{m=1}^\infty\sum_{|\alpha|=m}\left\|\partial^\alpha \nabla\partial_tw\right\|_{L^2(\Ome)}
            \epsilon^\alpha\tau^{(m-3)_+}M_{m}\right)\,ds,
        \end{split}
    \end{align}
where we used Minkowski's inequality in the last step of~\eqref{EQ44}. The trace estimate also produces lower-order boundary terms, which are
controlled by~$S_0$. Therefore, we arrive at
    \begin{equation}
            S_1\lec S_0+\int_0^t
	       (\|\partial_tw\|_{X_\tau}+\|\nabla\partial_tw\|_{X_\tau})\,ds.
   \llabel{EQ45}
     \end{equation}
For any $0\leq t\leq T_0$, since $\frac{\tau_0}{2}\leq \tau\leq\tau_0< \tau_1$, we obtain
    \begin{align}
        \begin{split}
            \label{EQ46}
            S_1\lec S_0+\int_0^t\|\partial_tw\|_{X_{\tau_1}(\Ome)}\,ds\lec S_0+t\|\partial_tw(0)\|_{X_{\tau_1}},
        \end{split}
    \end{align}
by Lemmas~\ref{L01} and~\ref{AL01}. The estimate for $S_0$ is obtained in the same way as~\eqref{EQ44}
and~\eqref{EQ46}, except that no $S_0$ term appears on the right-hand side
and the summation starts from $m=0$. Thus,
analogously,
    \begin{align}
        \label{EQ47}
        S_0\lec \int_0^t
	(\|\partial_tw\|_{X_\tau}+\|\nabla\partial_tw\|_{X_\tau})\,ds\lec
        &\|\partial_tw(0)\|_{X_{\tau_1}}.
    \end{align}
Combining \eqref{EQ46} and \eqref{EQ47}, we derive
    \begin{align}
    \label{EQ48}
    \|\partial_1\psi\|_{X_\tau}+\|\partial_2\psi\|_{X_\tau}+\|\partial_3\psi-1\|_{X_\tau}\lec t\|\partial_tw(0)\|_{X_{\tau_1}}
    ,
    \end{align}
which proves the assertion~(1).
    
\medskip
\noindent\textbf{Proof of (2).}
We choose 
\begin{align}
    \llabel{EQ49}
    T:= \bigl(2(C_1+1)(C_2+1)\|\partial_tw(0)\|_{X_{\tau_1}}\bigr)^{-1},
\end{align}
where $C_1$ is the constant provided in Lemma~\ref{L02}\,(1), while $C_2$ is provided in~\eqref{EQ143} from the Appendix. For any $0\leq t\leq T$, we obtain
\begin{equation}
    \|\partial_3\psi-1\|_{X_\tau}\leq (2C_1C_2)^{-1},
   \llabel{EQ50}
     \end{equation}
by the estimate~\eqref{EQ48}. Using Lemma~\ref{AL03}, we deduce
\begin{align}
\label{EQ51}
    \|(\partial_3\psi)^{-1}\|_{X_\tau}\lec 1.
\end{align}
Combining~\eqref{EQ51} with the definition~\eqref{EQ09}, we obtain the
estimate in the assertion~(2).

\medskip
\noindent\textbf{Proof of~(3).}
The function $\partial_t\psi$ satisfies the boundary value problem
\begin{align}
  \begin{cases}
    \Delta \partial_t\psi =0 & \text{in } \Omega_1,\\
\partial_t\psi(x_1,x_2,1,t)=\partial_tw(x_1,x_2,t) & \text{on } \Gat\times[0,T],\\
\partial_t\psi(x_1,x_2,2,t)=0 & \text{on } \Gamma_1\times[0,T].
\end{cases}  
\llabel{EQ52}
\end{align}
An analogous argument to~\eqref{EQ42} and~\eqref{EQ43} shows that
  \begin{align}
   \label{EQ53}
    \|\nabla\partial_t\psi\|_{X_\tau}\lec S_0^{(t)}+S_1^{(t)}
\end{align}
and
\begin{align}
    S_1^{(t)}\lec S_0^{(t)}+\sum_{m=0}^\infty\sum_{|\alpha'|=m}\sum_{i=1}^2\|\partial^{\alpha'}\partial_i\partial_tw\|
    _{H^\frac{1}{2}(\tau_t)}
    (\epsilon')^{\alpha'}\epsilon_i\tau^{(m-2)_+}M_{m+1},
\end{align}
where
\begin{equation}
S_k^{(t)}=\sum_{m=0}^\infty\sum_{|\alpha'|=m}\sum_{i=1}^3
    \|\partial^{\alpha'}\partial_3^k\partial_i\partial_t\psi\|\tau^{(m+k-3)_+}(\epsilon')^{\alpha'}\epsilon_3^k M_{m+k}
        \comma k=0,1.
   \llabel{EQ54}
     \end{equation}
By the trace theorem, Young's inequality, and the initial condition~\eqref{EQ02}, we derive
    \begin{align}
        \begin{split}
            \llabel{EQ55}
            S_1^{(t)}\lec&S_0^{(t)}+\sum_{m=0}^\infty\sum_{|\alpha'|=m}
            \sum_{i=1}^2
            \|\partial^{\alpha'}\partial_i\partial_tw\|_{L^2(\Ome)}^\frac{1}{2}
            \|\partial^{\alpha'}\partial_i\nabla\partial_t w\|_{L^2(\Ome)}^\frac{1}{2}
            (\epsilon')^{\alpha'}\epsilon_i\tau^{(m-2)_+}M_{m+1}
            \\\lec&
            S_0^{(t)}+
            \sum_{m=1}^\infty\sum_{|\alpha|=m}
            \left(
            \left\|\int_0^t\partial^\alpha \partial_{tt}w\,ds\right\|_{L^2(\Ome)}
            +\|\partial^\alpha\partial_tw(0)\|
            \right)
            \epsilon^\alpha\tau^{(m-3)_+}M_{m}
            \\&\indeq
            +\sum_{m=1}^\infty\sum_{|\alpha|=m}
            \left(
            \left\|\int_0^t\partial^\alpha\nabla \partial_{tt} w\,ds\right\|_{L^2(\Ome)}
            +\|\partial^\alpha\nabla\partial_tw(0)\|
            \right)
            \epsilon^\alpha\tau^{(m-3)_+}M_{m}
            \\\lec&
            S_0^{(t)}+
            \int_0^t\left(\sum_{m=1}^\infty\sum_{|\alpha|=m}\left\|\partial^\alpha \partial_{tt}w\right\|_{L^2(\Ome)}
            \epsilon^\alpha\tau^{(m-3)_+}M_{m}\right)\,ds
            \\&\indeq+
            \int_0^t\left(\sum_{m=1}^\infty\sum_{|\alpha|=m}\left\|\partial^\alpha \nabla\partial_{tt}w\right\|_{L^2(\Ome)}
            \epsilon^\alpha\tau^{(m-3)_+}M_{m}\right)\,ds
            \\&\indeq
            +\|\partial^\alpha\partial_tw(0)\|_{X_\tau}+\|\partial^\alpha\nabla\partial_tw(0)\|_{X_\tau}
         ,
        \end{split}
    \end{align}
    where the lower-order boundary terms are
controlled by~$S_0$.
Using Lemma~\ref{AL01} and the wave equation~\eqref{EQ01}, we obtain
    \begin{align}
    \begin{split}
    \label{EQ56}
    S_1^{(t)}&\lec S_0^{(t)}+\int_0^t\|\partial_{tt}w\|_{X_\tau}+\|\nabla\partial_{tt}w\|_{X_\tau}\,ds
    +\|\partial_tw(0)\|_{X_\tau}+\|\nabla\partial_tw(0)\|_{X_\tau}
    \\&=S_0^{(t)}+\int_0^t\|\Delta w\|_{X_\tau}+\|\nabla\Delta w\|_{X_\tau}\,ds
    +\|\partial_tw(0)\|_{X_\tau}+\|\nabla\partial_tw(0)\|_{X_\tau}
    \\&\lec S_0^{(t)}+\int_0^t\|\nabla w\|_{X_{\tau_1}}\,ds+\|\partial_tw(0)\|_{X_{\tau_1}}
    .
    \end{split}
    \end{align}
Similarly, $S_0^{(t)}$ satisfies 
\begin{align}\label{EQ57}
S_0^{(t)}\lec \int_0^t\|\nabla w\|_{X_{\tau_1}}\,ds+\|\partial_tw(0)\|_{X_{\tau_1}} .
\end{align}
Combining~\eqref{EQ53}, \eqref{EQ56}, and~\eqref{EQ57}, we derive
\begin{align}
    \|\nabla\partial_t\psi\|_{X_\tau}\lec\int_0^t\|\nabla w\|_{X_{\tau_1}}\,ds+\|\partial_tw(0)\|_{X_{\tau_1}}
    \lec
    \|\partial_tw(0)\|_{X_{\tau_1}}(1+t)
    .
   \llabel{EQ58}
\end{align}
Therefore, using the Poincar\'e inequality, we obtain
\begin{align}
\begin{split}
\|\partial_t\psi\|_{X_\tau}&
        =\sum_{m=0}^{\infty}\sum_{|\alpha|=m}\epsilon^\alpha 
        \tau^{(m-3)_+}M_m\|\partial^\alpha \partial_t\psi\|
        =\|\partial_t\psi\|+\sum_{m=1}^{\infty}\sum_{|\alpha|=m}\epsilon^\alpha 
        \tau^{(m-3)_+}M_m\|\partial^\alpha \partial_t\psi\|
        \\&\lec
        \|\partial_3\partial_t\psi\|+\sum_{m=0}^{\infty}\sum_{|\alpha|=m}
        \epsilon^\alpha 
        \tau^{(m-3)_+}M_m\|\partial^\alpha\nabla \partial_t\psi\|
        \underbrace{(\tau^{(m-2)_+-(m-3)_+}M_{m+1}M_m^{-1})}_{\lec 1}
        \\&
        \lec \|\nabla\partial_t\psi\|_{X_\tau}\lec \|\partial_tw(0)\|_{X_{\tau_1}}(1+t)
        \lec \|\partial_tw(0)\|_{X_{\tau_1}}(1+T)
	,
\end{split}
   \llabel{EQ59}
\end{align}
proving the assertion~(3).
\end{proof}

\section{Pressure estimate}
\label{sec04}
In this section, we complete the elliptic estimates for the pressure equation.

\begin{lemma}
\label{L03}
    Under the conditions of Theorem~\ref{T01}, we have
    \begin{align}
    \llabel{EQ60}
    \|\nabla q\|_{X_\tau}\lec (\|\partial_t w(0)\|_{X_{\tau_1}}+1)
    (\|v\|_{X_\tau}+\|v\|_{Y_\tau})(\|v\|_{X_\tau}+1)+\|\partial_t w(0)\|_{X_{\tau_1}}
    \end{align}
for any $0\leq t\leq T$.
\end{lemma}
\begin{proof}[Proof of lemma~\ref{L03}]
Applying $\partial_j(b_{ji}\cdot)$ to the Euler equation \eqref{EQ14}, we get
\begin{align}
    \label{EQ61}
    \partial_j(b_{ji}a_{ki}\partial_kq)=-\partial_j(b_{ji}\partial_tv_i)-
    \partial_j(b_{ji}v_ka_{mk}\partial_mv_i)+\partial_j(a_{ji}\psi_t\partial_3v_i).
\end{align}
By the Piola identity and the divergence-free condition, the first term on the right hand side of~\eqref{EQ61} can be rewritten as
\begin{align}
\label{EQ62}
    -\partial_j(b_{ji}\partial_tv_i)=\partial_j(\partial_tb_{ji}v_i)
\end{align}
since
\begin{equation}
\partial_t\partial_j(b_{ji}v_i)=0.
   \llabel{EQ63}
     \end{equation}
Substituting~\eqref{EQ62} into~\eqref{EQ61}, we derive
\begin{align}
    \begin{split}
        \label{EQ64}
        \Delta q=\partial_j(\delta_{jk}\partial_k q)=\partial_j(b_{ji}a_{ki}\partial_kq)
        +\partial_j\left(
        (\delta_{jk}-b_{ji}a_{ki})\partial_kq
        \right)
        =\partial_j f_j,
    \end{split}
\end{align}
where 
\begin{align}
\label{EQ65}
f_j=\partial_tb_{ji}v_{i}-b_{ji}v_ka_{mk}\partial_mv_i+a_{ji}\psi_t\partial_3v_i
+(\delta_{jk}-b_{ji}a_{ki})\partial_kq
.    
\end{align}
For the boundary condition, we apply $b_{3i}\cdot$ to each side of the
Euler equation obtaining
\eqref{EQ14}
\begin{equation}
b_{3i}a_{ki}\partial_k q=-b_{3i}\partial_tv_i
-b_{3i}v_ka_{mk}\partial_mv_i+a_{3i}\psi_t\partial_3v_i.
   \llabel{EQ66}
     \end{equation}
On the top boundary of the elastic body, $\Gat$, the kinematic condition \eqref{EQ18} implies
\begin{equation}
-b_{3i}\partial_tv_i=\partial_tb_{3i}v_i-\partial_{t}^2w
.
   \llabel{EQ67}
     \end{equation}
Therefore, on $\Gat$, we have
\begin{align}
    \begin{split}
        \llabel{EQ68}
        \partial_3q&=-b_{3i}\partial_tv_i
        -b_{3i}v_ka_{mk}\partial_mv_i+a_{3i}\psi_t\partial_3v_i
        +(\partial_3q-b_{3i}a_{ki}\partial_k q)
        \\&
        =\partial_tb_{3i}v_i-\partial_{t}^2w
        -b_{3i}v_ka_{mk}\partial_mv_i+a_{3i}\psi_t\partial_3v_i
        +(\partial_3q-b_{3i}a_{ki}\partial_k q)
        \\&
        =f_3-\partial_t^2w.
    \end{split}
\end{align}
On the far-top boundary of the fluid, by the definitions~\eqref{EQ06}, \eqref{EQ09}, and~\eqref{EQ11}, together with the boundary condition~\eqref{EQ17}, we obtain
\begin{align}
    \partial_3q&=-b_{3i}\partial_tv_i
        -b_{3i}v_ka_{mk}\partial_mv_i+a_{3i}\psi_t\partial_3v_i
        +(\partial_3q-b_{3i}a_{ki}\partial_k q)=0
\llabel{EQ69}
\end{align}
and
\begin{align}
\llabel{EQ70}
    f_3=\partial_tb_{3i}v_{i}-b_{3i}v_ka_{mk}\partial_mv_i+a_{3i}\psi_t\partial_3v_i
+(\delta_{3k}-b_{3i}a_{ki})\partial_kq=0.
\end{align}
Thus, the pressure $q$ satisfies the boundary value problem
\begin{align}
\llabel{EQ71}
    \begin{cases}
        &\Delta q=\partial_j f_j \onon{\Omega_1}
        \\&\partial_3q-f_3=-\partial_t^2w \onon{\Gat},
        \\&\partial_3q=f_3=0 \onon{\Gamma_1},
    \end{cases}
\end{align}
where $f_j$ is defined by~\eqref{EQ65}.
Integrating by parts, we obtain
\begin{align}
\llabel{EQ72}
\begin{split}
    \int_{\Omega_1} |\partial^{\alpha'}\nabla q|^2
    &=-\int_{\Gat}
    \partial^{\alpha'}\partial_3q\partial^{\alpha'}q
    -\int_{\Omega_1}\partial^{\alpha'}\Delta q\partial^{\alpha'}q
    \\&=-\int_{\Gat}
    \partial^{\alpha'}\partial_3q\partial^{\alpha'}q
    -\int_{\Omega_1}\partial^{\alpha'}\partial_jf_j\partial^{\alpha'}q
    \\&=-\int_{\Gat}\partial^{\alpha'}(f_3-\partial_t^2w)\partial^{\alpha'}q
    +\int_{\Omega_1}\partial^{\alpha'}f_j\partial^{\alpha'}\partial_jq
    +\int_{\Gat}\partial^{\alpha'}f_3\partial^{\alpha'}q
    \\&=\int_{\Gat}\partial^{\alpha'}\Delta w\partial^{\alpha'}q
    +\int_{\Omega_1}\partial^{\alpha'}f_j\partial^{\alpha'}\partial_jq.
\end{split}
\end{align}
By H\"older's, Young's, and the trace inequalities, we further derive
\begin{align}
    \label{EQ73}
    \|\partial^{\alpha'}\nabla q\|
    \lec
    \|\partial^{\alpha'}\Delta w\|
    +\|\partial^{\alpha'}\Delta\nabla w\|+\|\partial^{\alpha'}f_j\|.
\end{align}
For $k\in\mathbb{N}_0$, set
\begin{align}
    \llabel{EQ74}
    S^{(q)}_k=\sum_{m=0}^\infty\sum_{|\alpha'|=m}\sum_{i=1}^3
    \|\partial^{\alpha'}\partial_3^k\partial_i q\|
    \tau^{(m+k-3)_+}(\epsilon')^{\alpha'}\epsilon_3^k M_{m+k}.
\end{align}
Then for any $k\geq 2$, using \eqref{EQ64}, we obtain
\begin{align}
    \begin{split}
        \label{EQ75}
    S_k^{(q)}&\leq\sum_{m=0}^\infty\sum_{|\alpha'|=m}\sum_{i=1}^3\sum_{j=1,2}
    \|\partial^{\alpha'}\partial_3^{k-2}\partial_j^2\partial_i q\|
    \tau^{(m+k-3)_+}(\epsilon')^{\alpha'}\epsilon_j^2\epsilon_3^{k-2} M_{m+k}
    (\epsilon_3\epsilon_j^{-1})^2
    \\&\indeq
    +\sum_{m=0}^\infty\sum_{|\alpha'|=m}\sum_{i=1}^3
    \|\partial^{\alpha'}\partial_3^{k-2}\partial_i \partial_jf_j\|
    \tau^{(m+k-3)_+}(\epsilon')^{\alpha'}\epsilon_i\epsilon_j\epsilon_3^{k-2} M_{m+k}
    (\epsilon_3^2\epsilon_i^{-1}\epsilon_j^{-1})
    \\&\leq
    \frac{1}{2}S_{k-2}^{(q)}+S_{k}^{(f)},
    \end{split}
\end{align}
where\begin{equation}
S_k^{(f)}=\sum_{m=0}^\infty\sum_{|\alpha'|=m}
    \|\partial^{\alpha'}\partial_3^k f\|
    \tau^{(m+k-3)_+}(\epsilon')^{\alpha'}\epsilon_3^k M_{m+k}.
   \llabel{EQ76}
     \end{equation}
Summing~\eqref{EQ75} over all $k\geq 2$, we obtain
\begin{equation}
\sum_{k=0}^\infty S_k^{(q)}\leq S_0^{(q)}+S_1^{(q)}+\frac{1}{2}\sum_{k=0}^\infty S_k^{(q)}+
\sum_{k=2}^\infty S_k^{(f)}.
   \llabel{EQ77}
     \end{equation}
Therefore, we derive
\begin{align}
    \label{EQ78}
    \sum_{k=0}^\infty S_k^{(q)}\lec S_0^{(q)}+S_1^{(q)}+\|f\|_{X_\tau}.
\end{align}
We first bound $S_1^{(q)}$ in terms of $S_0^{(q)}$ and $f$,
\begin{align}
    \begin{split}
    \label{EQ79}
    S_1^{(q)}
    &=\sum_{m=0}^\infty\sum_{|\alpha'|=m}\sum_{i=1}^3
    \|\partial^{\alpha'}\partial_3 \partial_i q\|
    \tau^{(m-2)_+}(\epsilon')^{\alpha'}\epsilon_3 M_{m+1}
    \\&=\sum_{m=0}^\infty\sum_{|\alpha'|=m}\sum_{i=1}^2
    \|\partial^{\alpha'}\partial_i \partial_3 q\|
    \tau^{(m-2)_+}(\epsilon')^{\alpha'}\epsilon_i M_{m+1}
    (\epsilon_3\epsilon_i^{-1})
    \\&\indeq
    +\sum_{m=0}^\infty\sum_{|\alpha'|=m}\sum_{i=1}^2
    \|\partial^{\alpha'}\partial_i^2 q\|
    \tau^{(m-2)_+}(\epsilon')^{\alpha'}\epsilon_i M_{m+1}
    (\epsilon_3\epsilon_i^{-1})
    \\&\indeq
    +\sum_{m=0}^\infty\sum_{|\alpha'|=m}\sum_{i=1}^2
    \|\partial^{\alpha'}\partial_j f_j\|
    \tau^{(m-2)_+}(\epsilon')^{\alpha'}\epsilon_j M_{m+1}
    (\epsilon_3\epsilon_j^{-1})
    \\&
    \lec S_0^{(q)}+\|f\|_{X_\tau}
    .
    \end{split}
\end{align}
By \eqref{EQ73}, we obtain
\begin{align}
    \begin{split}
    \label{EQ80}
    S_0^{(q)}&=\sum_{m=0}^\infty\sum_{|\alpha'|=m}\sum_{i=1}^3
    \|\partial^{\alpha'}\partial_i q\|
    \tau^{(m-3)_+}(\epsilon')^{\alpha'} M_{m}
    \\&\leq
    \sum_{m=0}^\infty\sum_{|\alpha'|=m}(\|\partial^{\alpha'}\Delta w\|
    +\|\partial^{\alpha'}\Delta\nabla w\|+\|\partial^{\alpha'}f_j\|)
    \tau^{(m-3)_+}(\epsilon')^{\alpha'} M_{m}
    \\&\lec
    \|\nabla w\|_{X_\tau}+\|f\|_{X_\tau}
    \lec
    \|\partial_t w(0)\|_{X_{\tau_1}}+\|f\|_{X_\tau},
    \end{split}
\end{align}
where we used Lemma~\ref{AL01} in the last inequality. Combining \eqref{EQ78},
\eqref{EQ79}, and~\eqref{EQ80}, we get
\begin{align}
    \begin{split}
    \llabel{EQ81}
    \|\nabla q\|_{X_\tau}\lec \|\partial_t w(0)\|_{X_{\tau_1}}+\|f\|_{X_\tau}.
    \end{split}
\end{align}
Using the expression for $f_j$ in \eqref{EQ65} and Lemmas~\ref{AL01} and~\ref{AL02}, we  obtain
\begin{align}
    \begin{split}
        \|\nabla q\|_{X_\tau}&
        \lec \|\partial_t w(0)\|_{X_{\tau_1}}
        +\|\partial_tb_{ji}\|_{X_\tau}\|v_{i}\|_{X_\tau}+
        \|b_{ji}\|_{X_\tau}\|v_k\|_{X_\tau}\|a_{mk}\|_{X_\tau}
        (\|v_i\|_{X_\tau}+\|v_i\|_{Y_\tau})
        \\&\indeq
        +\|a_{ji}\|_{X_\tau}\|\psi_t\|_{X_\tau}(\|v_i\|_{X_\tau}+\|v_i\|_{Y_\tau})
        +\|\delta_{jk}-b_{ji}a_{ki}\|_{X_\tau}\|\nabla q\|_{X_\tau}
        \\&\lec \|\partial_t w(0)\|_{X_{\tau_1}}
        +(\|\partial_t w(0)\|_{X_{\tau_1}}+1)
        (\|v\|_{X_\tau}+1)(\|v_i\|_{X_\tau}+\|v_i\|_{Y_\tau})
       \\&\indeq
        +\|I-ba^T\|_{X_\tau}\|\nabla q\|_{X_\tau}.
    \end{split}
   \llabel{EQ82}
\end{align}
By Lemma~\ref{L02}, for $0\leq t\leq T$, we have 
\begin{align}
    \|I-ba^T\|_{X_\tau}\lec \|I-b\|_{X_\tau}+\|I-a^T\|_{X_\tau}+\|(I-b)(I-a^T)\|_{X_\tau}\lec
    \|\partial_tw(0)\|_{X_{\tau_1}}t.
   \llabel{EQ83}
\end{align}
Finally, we derive
\begin{align}
    \llabel{EQ84}
    \|\nabla q\|_{X_\tau}\lec (\|\partial_t w(0)\|_{X_{\tau_1}}+1)
    (\|v\|_{X_\tau}+\|v\|_{Y_\tau})(\|v\|_{X_\tau}+1)+\|\partial_t w(0)\|_{X_{\tau_1}},
\end{align}
for $0\leq t\leq T$, which concludes the proof of Lemma~\ref{L03}.
\end{proof}

\section{A priori estimate}
\label{sec05}
In this section, we prove the a priori estimate
on the velocity
by using the pressure estimate
and the product inequalities in analytic norms.
Applying $\partial^\alpha$ on each side of \eqref{EQ14}, testing by $\partial^\alpha v_i$, and
using H\"older's inequality, we obtain
\begin{align}
    \begin{split}
    \frac{1}{2}\frac{d}{dt}\|\partial^\alpha v\|^2\leq 
    \|\partial^\alpha (v_k a_{mk}\partial_m v)\|\|\partial^\alpha v\|
    +\|\partial^\alpha(J^{-1}\partial_t\psi\partial_3 v_i)\|\|\partial^\alpha v\|
    +\|\partial^\alpha(a_{ki}\partial_k q)\|\|\partial^\alpha v\|.
    \end{split}
   \llabel{EQ85}
\end{align}
Equivalently, we apply the $L^2$ energy estimate to each differentiated
equation, multiply it by the corresponding weight appearing in the definition
of the $X_\tau$-norm, and then sum over all multi-indices~$\alpha$. We get
\begin{align}
    \begin{split}
    \frac{d}{dt}\|v\|_{X_\tau}&=\sum_{m=0}^\infty\sum_{|\alpha|=m}\frac{d}{dt}
    \|\partial^\alpha v\|\epsilon^\alpha\tau^{(m-3)_+}M_m
    \\&\indeq+
    \dot\tau\sum_{m=4}^\infty\sum_{|\alpha|=m}\|\partial^\alpha v\|\epsilon^\alpha
    (m-3)\tau^{m-4}M_m
    \\&\lec
    \|v_ka_{mk}\partial_mv\|_{X_\tau}
    +\|J^{-1}\partial_t\psi \partial_3v\|_{X_\tau}
    +\|a_{ki}\partial_kq\|_{X_\tau}
    .
    \end{split}
   \llabel{EQ86}
\end{align}
By Lemmas~\ref{L02}, \ref{L03}, \ref{AL02}, and~\ref{AL03}, we derive
\begin{align}
    \begin{split}
    \frac{d}{dt}\|v\|_{X_\tau}&\lec 
    \dot\tau\|v\|_{Y_\tau}
    +\|v\|_{X_\tau}(\|\partial_t w(0)\|_{X_{\tau_1}}+1)
    (\|v\|_{X_\tau}+\|v\|_{Y_\tau})
    \\&\indeq
    +(\|\partial_t w(0)\|_{X_{\tau_1}}+1)^2
    (\|v\|_{X_\tau}+\|v\|_{Y_\tau})(\|v\|_{X_\tau}+1)
    \\&\indeq
    +\|\partial_t w(0)\|_{X_{\tau_1}}(\|\partial_t w(0)\|_{X_{\tau_1}}+1)
    \\&=
    (-M+(\|\partial_t w(0)\|_{X_{\tau_1}}+1)^2
    +(\|\partial_t w(0)\|_{X_{\tau_1}}+1)\|v\|_{X_\tau})\|v\|_{Y_\tau}
    \\&\indeq+
    (\|\partial_t w(0)\|_{X_{\tau_1}}+1)\|\partial_t w(0)\|_{X_{\tau_1}}
    +(\|\partial_t w(0)\|_{X_{\tau_1}}+1)^2(\|v\|_{X_\tau}+1)\|v\|_{X_\tau},
    \end{split}
   \llabel{EQ87}
\end{align}
which proves the desired a priori estimate in Theorem~\ref{T01}.

\section{Construction of solutions}
\label{sec06}
In this section, we prove the existence part of Theorem~\ref{T01} by an
iteration argument. Since both the ALE incompressibility constraint and the
kinematic boundary condition contain the time-dependent coefficients $b_{ji}$,
the next velocity $v^{(n+1)}$ and the pressure $q^{(n+1)}$ have to be
solved simultaneously. The pressure is therefore the Lagrange multiplier which
enforces the constraints at each step of the construction.
We first state the iterative construction precisely.

\begin{Proposition}[Iterative construction]\label{prop1}
Assume the hypotheses of Theorem~\ref{T01}. Let $A$ and $B$ be the
constants defined in Theorem~\ref{T01}. Choose a constant $C_0\geq1$,
depending only on the constants in the product estimates, the pressure estimate,
and the ALE coefficient estimates. Set
\begin{equation}
M:=10C_0(A+1)^3(A+1+B).
   \llabel{EQ88}
\end{equation}
Choose $T_*>0$ sufficiently small so that
\[
\tau_0-MT_*\geq \frac{\tau_0}{2},
\]
the estimates for the ALE coefficients in Lemma~\ref{L02} hold on
$[0,T_*]$, the perturbative pressure estimates of Section~\ref{sec04} hold
on $[0,T_*]$, and the fixed point argument below is valid on $[0,T_*]$.
Then set
\begin{equation}
T_0:=
\min\left\{
T_*,
\frac{\log(4/3)}{100C_0(A+B+1)^5}
\right\}.
   \llabel{EQ89}
\end{equation}
Let
\[
\tau(t)=\tau_0-Mt .
\]
Then there exists a sequence
$
\{(v^{(n)},q^{(n)})\}_{n\geq0}
$
such that
\begin{align}
  \begin{split}
&
v^{(n)}\in C([0,T_0];X_{\tau(t)})
\cap L^1(0,T_0;Y_{\tau(t)}),
\\&
\nabla q^{(n)}\in L^1(0,T_0;X_{\tau(t)}).
  \end{split}
   \llabel{EQ107}
     \end{align}
Moreover, for every $n\geq0$, the pair
$(v^{(n+1)},q^{(n+1)})$ solves
\begin{align}
    \begin{split}
    \llabel{EQ90}
    &\partial_tv_i^{(n+1)}
    +v_k^{(n)}a_{mk}\partial_mv_i^{(n)}
    -J^{-1}\psi_t\partial_3v_i^{(n)}
    +a_{ki}\partial_kq^{(n+1)}=0,
    \\&
    b_{ji}\partial_jv_i^{(n+1)}=0
     \inin{\Omega_1},
    \end{split}
\end{align}
with
\begin{align}
    v^{(0)}(t,x):=v_0(x),
    \andand
    v^{(n+1)}(0,x)=v_0(x),
   \llabel{EQ91}
\end{align}
and
  \begin{align}
  \begin{split}
   &
       v_3^{(n+1)}=0
    \onon{\Gamma_1},
   \\&
    b_{3i}v_i^{(n+1)}=\partial_tw
    \onon{\Gat} .
  \end{split}
   \label{EQ92}
  \end{align}
The sequence satisfies
\begin{align}
    \sup_{t\in[0,T_0]}\|v^{(n)}(t)\|_{X_{\tau(t)}}
    +
    M\int_0^{T_0}\|v^{(n)}(s)\|_{Y_{\tau(s)}}\,ds
    \leq B.
   \label{EQ93}
\end{align}
Finally, $\{v^{(n)}\}$ is Cauchy in
$C([0,T_0];X_{\tau(t)})$ and~$L^1(0,T_0;Y_{\tau(t)})$. Its limit is the
solution asserted in Theorem~\ref{T01}.
\end{Proposition}

\begin{proof}
We divide the proof into four parts.

\smallskip
\noindent
{\bf Step 1. Solvability of one iterative step.}
Assume that $v^{(n)}$ has already been constructed and satisfies
\[
v^{(n)}\in C([0,T_0];X_{\tau(t)})
\cap L^1(0,T_0;Y_{\tau(t)}),
\]
together with the constraints and boundary conditions
  \begin{align}
  \begin{split}
   &
b_{ji}\partial_jv_i^{(n)}=0
\inin{\Omega_1},
  \\&
v_3^{(n)}=0
\onon{\Gamma_1},
   \\&
b_{3i}v_i^{(n)}=\partial_tw
\onon{\Gat} .
  \end{split}
   \llabel{EQ116}
  \end{align}
We prove that the $(n+1)$-th system is well-defined in the same analytic
class.
Let $U\in C([0,T_0];X_{\tau(t)})$ be a candidate for~$v^{(n+1)}$. For this
fixed $U$, define $q[U]$, with normalization
\[
\int_{\Omega_1}q[U]\,dx=0,
\]
as the solution of
\begin{align}
\label{EQ94}
    \begin{cases}
        \Delta q[U]=\partial_j f_j[U]
        & \text{in }\Omega_1,
        \\[2mm]
        \partial_3q[U]-f_3[U]=-\partial_t^2w
        & \text{on }\Gat,
        \\[2mm]
        \partial_3q[U]=f_3[U]=0
        & \text{on }\Gamma_1,
    \end{cases}
\end{align}
where
\begin{align}
\llabel{EQ95}
    f_j[U]
    =
    \partial_tb_{ji}U_i
    -b_{ji}v_k^{(n)}a_{mk}\partial_mv_i^{(n)}
    +a_{ji}\psi_t\partial_3v_i^{(n)}
    +(\delta_{jk}-b_{ji}a_{ki})\partial_kq[U].
\end{align}
The Neumann compatibility condition follows from the flux compatibility
condition in Theorem~\ref{T01}. Indeed,
\[
\int_{\Gat}\partial_t^2 w\,dx'
=
\frac{d}{dt}\int_{\Gat}\partial_t w\,dx'
=0.
\]
We solve \eqref{EQ94} perturbatively. The only term in $f_j[U]$ containing
$\nabla q[U]$ is
\[
(\delta_{jk}-b_{ji}a_{ki})\partial_kq[U].
\]
By the choice of $T_*$, the quantity $\|I-ba^T\|_{X_{\tau(t)}}$ is
sufficiently small on $[0,T_0]$. Hence, the perturbative elliptic estimate
from Section~\ref{sec04} gives a unique
\[
\nabla q[U]\in L^1(0,T_0;X_{\tau(t)}),
\]
and
\begin{align}
\label{EQ96}
    \begin{split}
    \|\nabla q[U]\|_{X_{\tau(t)}}
    &\leq
    C_0(A+1)\|U\|_{X_{\tau(t)}}+C_0A
    \\&\indeq
    +C_0(A+1)
    \bigl(
    \|v^{(n)}\|_{X_{\tau(t)}}+\|v^{(n)}\|_{Y_{\tau(t)}}
    \bigr)
    \|v^{(n)}\|_{X_{\tau(t)}} .
    \end{split}
\end{align}
Moreover, subtracting the elliptic problems corresponding to two candidates
$U,V\in C([0,T_0];X_{\tau(t)})$, the same perturbative estimate gives
\begin{align}
\label{EQ97}
    \|\nabla(q[U]-q[V])\|_{X_{\tau(t)}}
    \leq
    C_0(A+1)\|U-V\|_{X_{\tau(t)}}.
\end{align}
Define
\begin{align}
\llabel{EQ98}
    \Phi(U)_i(t)
    :=
   v_{0i}
    -
    \int_0^t
    \left(
    v_k^{(n)}a_{mk}\partial_mv_i^{(n)}
    -J^{-1}\psi_t\partial_3v_i^{(n)}
    +a_{ki}\partial_kq[U]
    \right)(s)\,ds .
\end{align}
By the product estimates and \eqref{EQ96}, the integrand belongs to
$L^1(0,T_0;X_{\tau(t)})$. Hence,
\[
\Phi(U)\in C([0,T_0];X_{\tau(t)}).
\]
Using \eqref{EQ97}, we obtain
\[
\|\Phi(U)(t)-\Phi(V)(t)\|_{X_{\tau(t)}}
\leq
C_0(A+1)
\int_0^t\|U(s)-V(s)\|_{X_{\tau(s)}}\,ds.
\]
Therefore,
\[
\sup_{t\in[0,T_0]}
\|\Phi(U)(t)-\Phi(V)(t)\|_{X_{\tau(t)}}
\leq
C_0(A+1)T_0
\sup_{t\in[0,T_0]}
\|U(t)-V(t)\|_{X_{\tau(t)}}.
\]
By the choice of $T_0$, we have
\[
C_0(A+1)T_0\leq \frac12
,
\]
showing that
$\Phi$ is a contraction on $C([0,T_0];X_{\tau(t)})$. Let
$v^{(n+1)}$ denote its fixed point, and set
\[
q^{(n+1)}:=q[v^{(n+1)}].
\]
Then $v^{(n+1)}\in C([0,T_0];X_{\tau(t)})$,
$\nabla q^{(n+1)}\in L^1(0,T_0;X_{\tau(t)})$ and
\[
\partial_tv_i^{(n+1)}
+
v_k^{(n)}a_{mk}\partial_mv_i^{(n)}
-
J^{-1}\psi_t\partial_3v_i^{(n)}
+
a_{ki}\partial_kq^{(n+1)}
=0.
\]
Moreover, $q^{(n+1)}$ solves
\begin{align}
\label{EQ99}
    \begin{cases}
        \Delta q^{(n+1)}=\partial_j f_j^{(n+1)}
        & \text{in }\Omega_1,
        \\[2mm]
        \partial_3q^{(n+1)}-f_3^{(n+1)}=-\partial_t^2w
        & \text{on }\Gat,
        \\[2mm]
        \partial_3q^{(n+1)}=f_3^{(n+1)}=0
        & \text{on }\Gamma_1,
    \end{cases}
\end{align}
where
\begin{align}
\label{EQ100}
    f_j^{(n+1)}
    =
    \partial_tb_{ji}v_i^{(n+1)}
    -b_{ji}v_k^{(n)}a_{mk}\partial_mv_i^{(n)}
    +a_{ji}\psi_t\partial_3v_i^{(n)}
    +(\delta_{jk}-b_{ji}a_{ki})\partial_kq^{(n+1)}.
\end{align}
We now verify that the constraints and boundary conditions are preserved. Using
the evolution equation and \eqref{EQ99}--\eqref{EQ100}, we have
\[
\begin{aligned}
\partial_t\partial_j(b_{ji}v_i^{(n+1)})
&=
\partial_j(\partial_tb_{ji}v_i^{(n+1)})
+
\partial_j(b_{ji}\partial_tv_i^{(n+1)})
\\
&=
\partial_j\left(
\partial_tb_{ji}v_i^{(n+1)}
-b_{ji}v_k^{(n)}a_{mk}\partial_mv_i^{(n)}
+a_{ji}\psi_t\partial_3v_i^{(n)}
-b_{ji}a_{ki}\partial_kq^{(n+1)}
\right)
=0.
\end{aligned}
\]
Since $b(\cdot,0)=I$, the compatibility condition \eqref{EQ28} gives
\[
\partial_j(b_{ji}v_i^{(n+1)})(0)=\partial_i(v_0)_i=0.
\]
Therefore,
\[
\partial_j(b_{ji}v_i^{(n+1)})=0
\inin{\Omega_1}.
\]
By the Piola identity, this is equivalent to
\[
b_{ji}\partial_jv_i^{(n+1)}=0
\inin{\Omega_1}.
\]
On $\Gat$, using the evolution equation and the boundary condition in
\eqref{EQ99}, we obtain
\[
\begin{aligned}
\partial_t(b_{3i}v_i^{(n+1)}-\partial_tw)
&=
\partial_tb_{3i}v_i^{(n+1)}
+
b_{3i}\partial_tv_i^{(n+1)}
-\partial_t^2w
\\
&=
\partial_tb_{3i}v_i^{(n+1)}
-b_{3i}v_k^{(n)}a_{mk}\partial_mv_i^{(n)}
+a_{3i}\psi_t\partial_3v_i^{(n)}
-b_{3i}a_{ki}\partial_kq^{(n+1)}
-\partial_t^2w
=0.
\end{aligned}
\]
Again by the compatibility condition \eqref{EQ28},
\[
b_{3i}(\cdot,0) v_{0i}-\partial_tw(\cdot,0)=0
\onon{\Gat}.
\]
Hence,
\[
b_{3i}v_i^{(n+1)}=\partial_tw
\onon{\Gat} .
\]
Finally, on $\Gamma_1$, we have
\[
\psi_t=0,\qquad
\partial_1\psi=\partial_2\psi=0,\qquad
v_3^{(n)}=0,
\quad
\text{and}
\quad
\partial_3q^{(n+1)}=0.
\]
Since $v_3^{(n)}=0$ on $\Gamma_1$, its tangential derivatives vanish
there. Thus the third component of the evolution equation gives
\[
\partial_tv_3^{(n+1)}=0
\onon{\Gamma_1}.
\]
Since $v_{03}=0$ on $\Gamma_1$, we get
\[
v_3^{(n+1)}=0
\onon{\Gamma_1}
.
\]
It remains to obtain the $L^1_tY_{\tau(t)}$ regularity. Applying the analytic
energy estimate to the $(n+1)$-th linear system and using \eqref{EQ96} with
$U=v^{(n+1)}$, we get
\begin{align}
\label{EQ101}
\begin{split}
&
\frac{d}{dt}\|v^{(n+1)}\|_{X_{\tau}}
+
M\|v^{(n+1)}\|_{Y_{\tau}}
\\&\indeq
\leq
C_0(A+1)\|v^{(n+1)}\|_{X_{\tau}}
+
C_0(A+1)(A+1+\|v^{(n)}\|_{X_{\tau}})
\|v^{(n)}\|_{Y_{\tau}}
\\&\indeq\indeq
+
C_0(A+1)^2
(\|v^{(n)}\|_{X_{\tau}}+1)\|v^{(n)}\|_{X_{\tau}}
+
C_0(A+1)^2 .
\end{split}
\end{align}
The right-hand side belongs to $L^1(0,T_0)$ by the induction hypothesis.
Integrating \eqref{EQ101} in time gives
\[
v^{(n+1)}\in L^1(0,T_0;Y_{\tau(t)})
,
\]
which proves the solvability of the $(n+1)$-th system in the desired class.

\smallskip
\noindent
{\bf Step 2. Uniform estimates.}
Set
  \begin{align}
  \begin{split}
   &
X_n(t):=\|v^{(n)}(t)\|_{X_{\tau(t)}},
  \\&
Y_n(t):=\|v^{(n)}(t)\|_{Y_{\tau(t)}}
\end{split}
   \llabel{EQ115}
  \end{align}
and
\[
B_n:=
\sup_{t\in[0,T_0]}X_n(t)
+
M\int_0^{T_0}Y_n(s)\,ds .
\]
From \eqref{EQ101}, we obtain
\begin{align}
\label{EQ102}
\begin{split}
\frac{d}{dt}X_{n+1}(t)+MY_{n+1}(t)
&\leq
C_0(A+1)X_{n+1}(t)
+
C_0(A+1)(A+1+X_n(t))Y_n(t)
\\&\indeq
+
C_0(A+1)^2(X_n(t)+1)X_n(t)
+
C_0(A+1)^2 .
\end{split}
\end{align}
We first note that~$MT_0\leq1$. Hence,
\[
B_0
=
\sup_{t\in[0,T_0]}\|v_0\|_{X_{\tau(t)}}
+
M\int_0^{T_0}\|v_0\|_{Y_{\tau(t)}}\,dt
\leq
\|v_0\|_{X_{\tau_0}}+\|v_0\|_{Y_{\tau_0}}
\leq \frac14B .
\]
Assume that $B_n\leq B$. By Gronwall's inequality applied to
\eqref{EQ102}, we get 
\[
\begin{aligned}
&
X_{n+1}(t)+M\int_0^tY_{n+1}(s)\,ds
\\&\indeq
\leq
e^{C_0(A+1)T_0}
\bigg(
X_0
+
\int_0^{T_0}
C_0(A+1)(A+1+X_n(s))Y_n(s)\,ds
\\&\indeq\indeq
+
\int_0^{T_0}
\left(
C_0(A+1)^2(X_n(s)+1)X_n(s)
+
C_0(A+1)^2
\right)ds
\bigg)
\end{aligned}
\]
for $0\leq t\leq T_0$,
By the definition of $M$,
\[
\int_0^{T_0}
C_0(A+1)(A+1+X_n(s))Y_n(s)\,ds
\leq
C_0(A+1)(A+1+B)M^{-1}B
\leq
\frac1{10}B.
\]
By the definition of $T_0$, and since $B\geq4$,
\[
\int_0^{T_0}
\left(
C_0(A+1)^2(X_n(s)+1)X_n(s)
+
C_0(A+1)^2
\right)ds
\leq
C_0(A+1)^2(B+1)^2T_0
\leq
\frac1{10}B.
\]
Also,
\[
e^{C_0(A+1)T_0}\leq \frac43.
\]
Therefore,
\[
X_{n+1}(t)+M\int_0^tY_{n+1}(s)\,ds
\leq
\frac43
\left(
\frac14B+\frac1{10}B+\frac1{10}B
\right)
\leq
\frac35B.
\]
Taking the supremum over $t\in[0,T_0]$, we obtain
\[
B_{n+1}\leq B.
\]
By induction,
\[
B_n\leq B
    \comma n\geq0
    ,
\]
proving~\eqref{EQ93}.

\smallskip
\noindent
{\bf Step 3. Difference estimate and summability.}
Set
\[
\widetilde v^{(n+1)}:=v^{(n+1)}-v^{(n)}
\andand
\widetilde q^{(n+1)}:=q^{(n+1)}-q^{(n)}.
\]
Subtracting the equations for $v^{(n+1)}$ and $v^{(n)}$, we get
\begin{align}
\label{EQ103}
\begin{split}
&\partial_t\widetilde v_i^{(n+1)}
+
\widetilde v_k^{(n)}a_{mk}\partial_mv_i^{(n)}
+
v_k^{(n-1)}a_{mk}\partial_m\widetilde v_i^{(n)}
-
J^{-1}\psi_t\partial_3\widetilde v_i^{(n)}
+
a_{ki}\partial_k\widetilde q^{(n+1)}
=0,
\end{split}
\end{align}
with
\begin{align}
  \begin{split}
&
b_{ji}\partial_j\widetilde v_i^{(n+1)}=0
\inin{\Omega_1}
,
\\&
\widetilde v^{(n+1)}(0)=0
\inin{\Omega_1}
\end{split}
   \llabel{EQ108}
     \end{align}
and
  \begin{align}
  \begin{split}
   &
\widetilde v_3^{(n+1)}=0
\onon{\Gamma_1},
\\&
b_{3i}\widetilde v_i^{(n+1)}=0
\onon{\Gat} .
  \end{split}
   \llabel{EQ109}
  \end{align}
The corresponding pressure estimate is
\begin{align}
\label{EQ104}
\begin{split}
\|\nabla\widetilde q^{(n+1)}\|_{X_{\tau}}
&\leq
C_0(A+1)\|\widetilde v^{(n+1)}\|_{X_{\tau}}
\\&\indeq
+
C_0(A+1)^2
\|\widetilde v^{(n)}\|_{X_{\tau}}
\bigl(\|v^{(n)}\|_{X_{\tau}}+\|v^{(n)}\|_{Y_{\tau}}\bigr)
\\&\indeq
+
C_0(A+1)^2
\|v^{(n-1)}\|_{X_{\tau}}
\bigl(
\|\widetilde v^{(n)}\|_{X_{\tau}}
+
\|\widetilde v^{(n)}\|_{Y_{\tau}}
\bigr).
\end{split}
\end{align}
Applying the analytic energy estimate to \eqref{EQ103} and using
\eqref{EQ104}, we obtain
\begin{align}
\label{EQ105}
\begin{split}
\frac{d}{dt}\|\widetilde v^{(n+1)}\|_{X_{\tau}}
+
M\|\widetilde v^{(n+1)}\|_{Y_{\tau}}
&\leq
C_0(A+1)\|\widetilde v^{(n+1)}\|_{X_{\tau}}
\\&\indeq
+
C_0(A+1)^3
\|\widetilde v^{(n)}\|_{X_{\tau}}
\bigl(\|v^{(n)}\|_{X_{\tau}}+\|v^{(n)}\|_{Y_{\tau}}\bigr)
\\&\indeq
+
C_0(A+1)^3
(\|v^{(n-1)}\|_{X_{\tau}}+1)
\bigl(
\|\widetilde v^{(n)}\|_{X_{\tau}}
+
\|\widetilde v^{(n)}\|_{Y_{\tau}}
\bigr).
\end{split}
\end{align}
Define
\[
\widetilde X_n(t):=\|\widetilde v^{(n)}(t)\|_{X_{\tau(t)}}
\andand
\widetilde Y_n(t):=\|\widetilde v^{(n)}(t)\|_{Y_{\tau(t)}}.
\]
Also, let
\[
\widetilde B_n
:=
\sup_{t\in[0,T_0]}\widetilde X_n(t)
+
M\int_0^{T_0}\widetilde Y_n(s)\,ds.
\]
Using $B_n\leq B$ and Gronwall's inequality in \eqref{EQ105}, we find
\[
\widetilde B_{n+1}
\leq
\frac83
\max\{M^{-1},T_0\}C_0(A+1)^3B\,\widetilde B_n.
\]
By the choices of $M$ and $T_0$,
\[
\frac83
\max\{M^{-1},T_0\}C_0(A+1)^3B
\leq
\frac4{15}.
\]
Hence,
 \begin{equation}
  \widetilde B_{n+1}\leq \frac4{15}\widetilde B_n.
  \label{EQ106}
 \end{equation}
Since $\widetilde B_1<\infty$, \eqref{EQ106} implies
\[
\sum_{n=1}^{\infty}
\sup_{t\in[0,T_0]}\|\widetilde v^{(n)}(t)\|_{X_{\tau(t)}}<\infty
\]
and
\[
\sum_{n=1}^{\infty}
\int_0^{T_0}\|\widetilde v^{(n)}(t)\|_{Y_{\tau(t)}}\,dt<\infty.
\]
Therefore, $\{v^{(n)}\}$ is Cauchy in
$C([0,T_0];X_{\tau(t)})$ and in $L^1(0,T_0;Y_{\tau(t)})$. Thus there
exists
\[
v\in C([0,T_0];X_{\tau(t)})
\cap L^1(0,T_0;Y_{\tau(t)})
\]
such that
\[
v^{(n)}\to v
\inin{C([0,T_0];X_{\tau(t)})}
\]
and
\[
v^{(n)}\to v
\inin{L^1(0,T_0;Y_{\tau(t)})}.
\]
The pressure differences are also summable. Indeed, integrating
\eqref{EQ104} in time and using $B_n\leq B$, we get
\[
\sum_{n=1}^{\infty}
\int_0^{T_0}
\|\nabla\widetilde q^{(n)}(t)\|_{X_{\tau(t)}}\,dt<\infty.
\]
Hence, there exists
\[
\nabla q\in L^1(0,T_0;X_{\tau(t)})
\]
such that
\[
\nabla q^{(n)}\to \nabla q
\inin{L^1(0,T_0;X_{\tau(t)})}.
\]

\smallskip
\noindent
{\bf Step 4. Passage to the limit.}
We pass to the limit in the iterative equation. Since $a$, $b$, $J$, and
$\psi_t$ are fixed functions determined by $w$, the linear coefficient
terms pass to the limit directly. For the transport term, the product estimate
gives
\[
\begin{aligned}
&\int_0^{T_0}
\left\|
v_k^{(n)}a_{mk}\partial_mv_i^{(n)}
-
v_ka_{mk}\partial_mv_i
\right\|_{X_{\tau(t)}}\,dt
\\&\indeq
\leq
C_0
\sup_{t\in[0,T_0]}\|v^{(n)}(t)-v(t)\|_{X_{\tau(t)}}
\int_0^{T_0}\|v^{(n)}(t)\|_{Y_{\tau(t)}}\,dt
\\&\indeq\indeq
+
C_0
\sup_{t\in[0,T_0]}\|v(t)\|_{X_{\tau(t)}}
\int_0^{T_0}\|v^{(n)}(t)-v(t)\|_{Y_{\tau(t)}}\,dt,
\end{aligned}
\]
which tends to zero. Similarly,
\[
J^{-1}\psi_t\partial_3v^{(n)}
\to
J^{-1}\psi_t\partial_3v
\inin{L^1(0,T_0;X_{\tau(t)})}
\]
and
\[
a_{ki}\partial_kq^{(n+1)}
\to
a_{ki}\partial_kq
\inin{L^1(0,T_0;X_{\tau(t)})}.
\]
Passing to the limit in the integral formulation of the iterative equation, we obtain
\[
v_i(t)=v_{0i}-\int_0^t
\left(v_k a_{mk}\partial_mv_i-J^{-1}\psi_t\partial_3v_i+a_{ki}\partial_kq\right)(s)\,ds .
\]
Equivalently, $v$ satisfies \eqref{EQ14} in the sense of distributions in time; since the integrand belongs to
$L^1(0,T_0;X_{\tau(t)})$, the equation holds in the corresponding $L^1$-sense.
Thus $v$ satisfies
\[
\partial_tv_i
+
v_ka_{mk}\partial_mv_i
-
J^{-1}\psi_t\partial_3v_i
+
a_{ki}\partial_kq
=0
\]
in $L^1(0,T_0;X_{\tau(t)})$, with $v(0)=v_0$.
Since each $v^{(n)}$ satisfies
  \begin{align}
  \begin{split}
   &b_{ji}\partial_jv_i^{(n)}=0,
   \\&
v_3^{(n)}=0
\onon{\Gamma_1},
   \\&
   b_{3i}v_i^{(n)}=\partial_tw\onon{\Gat},   
\end{split}
   \llabel{EQ111}
  \end{align}
we may pass to the limit and obtain
  \begin{align}
  \begin{split}
   &
   b_{ji}\partial_jv_i=0
   \inin{\Omega_1},
   \\&
   v_3=0\onon{\Gamma_1},   
   \\&
   b_{3i}v_i=\partial_tw\onon{\Gat}
   .   
  \end{split}
   \llabel{EQ110}
  \end{align}
Therefore, $v$ solves \eqref{EQ14} with boundary conditions \eqref{EQ17} and~\eqref{EQ18}. The uniform estimate \eqref{EQ93} passes to the limit by Fatou's
lemma, yielding~\eqref{EQ31}.
It remains only to record uniqueness. Let $v$ and $\bar v$ be two solutions
in the class
\[
C([0,T_0];X_{\tau(t)})\cap L^1(0,T_0;Y_{\tau(t)})
\]
with the same initial data. Applying the same difference estimate as above to
$v-\bar v$ and the corresponding pressure difference gives
\[
\sup_{t\in[0,T]}\|v(t)-\bar v(t)\|_{X_{\tau(t)}}
\leq
C_0\int_0^T
\|v(t)-\bar v(t)\|_{X_{\tau(t)}}\,dt
\]
for $T\leq T_0$, after absorbing the $Y_\tau$-term by the decreasing
analytic radius. Since $v(0)=\bar v(0)$, Gronwall's inequality gives
$v=\bar v$ on a sufficiently small time interval. Repeating the argument
finitely many times gives uniqueness on $[0,T_0]$.

This completes the proof of Proposition~\ref{prop1} and thus also of Theorem~\ref{T01}.
\end{proof}

\section{A mismatch counterexample}
\label{sec07}
In this section, we provide an example showing that the pressure matching condition
is not preserved by the Euler evolution
in analytic classes, even when it holds initially and the
kinematic boundary condition is imposed.
Setting
$w_t(\cdot,0)=0$ implies $w\equiv 0$. Then the kinematic condition reduces to
$v_3=0$ on $\partial\Omega_1$, and the fluid equations read
\begin{align}
    \begin{split}
    \llabel{EQ119}
    &v_t+v\cdot\nabla v+\nabla q=0  ,
    \\&
    \nabla\cdot v=0 \inin{\Omega_1},
    \\&
    v_3=0 \onon{\partial\Omega_1}
    .
    \end{split}
\end{align}
Choose 
\begin{align}
v_0(x)=
\left(
\sin x_2,\,
\left(\frac{\pi^2+2}{2}\cos(\pi(x_3-1))-1\right)\sin x_1,\,
0
\right)
   \llabel{EQ120}
\end{align}
and
\begin{equation}
q_0(x)=
\bigl(\cos(\pi(x_3-1))-1\bigr)\cos x_1\cos x_2.
   \llabel{EQ121}
     \end{equation}
A direct computation shows that $v_0$ is divergence-free, $v_3(\cdot,0)=0$ on
$\partial\Omega_1$, and $q_0$ satisfies the corresponding Neumann pressure
problem at $t=0$.
By the Euler equation,
\begin{align}
  \begin{split}
    v_t(0)=-v_0\cdot\nabla v_0-\nabla q_0.
  \end{split}
   \llabel{EQ122}
     \end{align}
A direct computation gives
\begin{align}
  \begin{split}
v_t(0,x)&=
\biggl(
-\frac{\pi^2}{2}\cos(\pi(x_3-1))\sin x_1\cos x_2,\,
-\frac{\pi^2}{2}\cos(\pi(x_3-1))\cos x_1\sin x_2,\,
\\&\indeq\indeq\indeq\indeq\indeq\indeq\indeq\indeq\indeq\indeq\indeq\indeq\indeq\indeq\indeq
\pi\sin(\pi(x_3-1))\cos x_1\cos x_2
\biggr).
  \end{split}
   \llabel{EQ123}
     \end{align}
The pressure satisfies
\begin{equation}
\Delta q=-\sum_{i,j=1}^3 \partial_i v_j\,\partial_j v_i.
   \llabel{EQ124}
     \end{equation}
Therefore,
\begin{equation}
\Delta q_t(0)
=
-2\sum_{i,j=1}^3
\partial_i v_{0j}\,\partial_j(v_t(0))_i.
   \llabel{EQ125}
     \end{equation}
Hence,
\begin{equation}
\begin{aligned}
\Delta q_t(0)
={}&
-\frac{\pi^2}{2}\cos(\pi(x_3-1))\sin x_1\sin(2x_2)
\\
&-
\frac{\pi^2}{2}
\left(\frac{\pi^2+2}{2}\cos(\pi(x_3-1))-1\right)
\cos(\pi(x_3-1))
\sin(2x_1)\sin x_2
\\
&-
\frac{\pi^2(\pi^2+2)}{2}
\sin^2(\pi(x_3-1))
\sin(2x_1)\sin x_2.
\end{aligned}
   \label{EQ126}
     \end{equation}
More precisely, $q_t(0)$ solves the Neumann problem
\begin{equation}
\begin{cases}
\Delta q_t(0)=
-2\displaystyle\sum_{i,j=1}^3
\partial_i v_{0j}\,\partial_j(v_t(0))_i,
& x\in \mathbb T^2\times(1,2),\\[6pt]
\partial_3 q_t(0)=0,
& x_3=1,2.
\end{cases}
   \llabel{EQ127}
     \end{equation}
We may also impose the normalization
\begin{equation}
\int_{\mathbb T^2\times(1,2)}q_t(0,x)\,dx=0.
   \llabel{EQ128}
     \end{equation}
From the above computation~\eqref{EQ126}, the right-hand side of $\Delta q_t$ contains the horizontal Fourier mode
\begin{equation}
-\frac{\pi^2}{2}\cos(\pi(x_3-1))\sin x_1\sin(2x_2).
   \llabel{EQ129}
     \end{equation}
Since the domain is the flat channel $\mathbb T^2\times(1,2)$, the Neumann Green operator diagonalizes with respect to the horizontal Fourier modes. Hence, the corresponding part of $q_t(0)$ has the form
\begin{equation}
\beta(x_3)\sin x_1\sin(2x_2),
   \llabel{EQ130}
     \end{equation}
where $\beta$ solves
\begin{equation}
\begin{cases}
\beta''(x_3)-5\beta(x_3)
=
-\dfrac{\pi^2}{2}\cos(\pi(x_3-1)),
& 1<x_3<2,\\[6pt]
\beta'(1)=\beta'(2)=0.
\end{cases}
   \llabel{EQ131}
     \end{equation}
Since $\cos(\pi(x_3-1))$ is itself a Neumann eigenfunction on $(1,2)$, the Neumann Green operator maps this source to a nonzero multiple of the same function. In particular,
\begin{equation}
\beta(1)\neq 0.
   \llabel{EQ132}
     \end{equation}
Consequently, $q_t(0)$ contains a nonzero boundary mode
\begin{equation}
c\sin x_1\sin(2x_2)
   \onon{x_3=1}
   \llabel{EQ133}
   ,
     \end{equation}
where $c\neq 0$.
Therefore,
\begin{equation}
q_t(0,x_1,x_2,1)\not\equiv 0.
   \llabel{EQ134}
     \end{equation}
Since
\begin{equation}
q_0(x_1,x_2,1)=0,
   \llabel{EQ135}
     \end{equation}
we obtain by Taylor expansion
\begin{equation}
q(t,x_1,x_2,1)
=
tq_t(0,x_1,x_2,1)+O(t^2).
   \llabel{EQ136}
\end{equation}
Since the pressure is determined only up to an additive function of time, one
may replace $q(t,x)$ by $q(t,x)+c(t)$. However, this freedom only changes
the spatially constant mode of the trace on $\mathbb T^2\times\{1\}$. The
nonzero boundary mode
\begin{equation}
c\sin x_1\sin(2x_2),
%,\qquad c\neq 0,
   \llabel{EQ137}
     \end{equation}
where $c\neq0$,
cannot be removed by adding a function of time. Therefore, for sufficiently
small $t>0$, there is no choice of $c(t)$ such that
\begin{equation}
q(t,x_1,x_2,1)+c(t)\equiv 0
\onon{\mathbb T^2}
.
   \llabel{EQ138}
     \end{equation}
Thus the pressure matching condition cannot be imposed on
$\mathbb T^2\times\{1\}$, even modulo the usual additive time-dependent
constant in the pressure.

\appendix
\section{Analyticity properties}\label{Appendix}
This appendix is devoted to auxiliary estimates in analytic norms. The
proof of Lemma~\ref{AL02} is given in~\cite{KOS} (see also \cite{JKL1,JKL2}). For completeness, we
provide the proofs of the remaining estimates, which are adapted to the
particular analytic norm used in this paper. 

\begin{lemma}[Loss of analytic radius under differentiation]
\label{AL01}
Let $0<c_0<\tilde\tau<\tau$. Then, for each coordinate derivative
$\partial_j$, with $j\in\{1,2,3\}$, one has
\begin{align}
\label{EQ139}
    \|\partial_j f\|_{X_{\widetilde\tau}}
    \lec_{c_0,\tilde\tau,\tau}
    \|f\|_{X_\tau}.
\end{align}
Consequently, for any multi-index $\alpha\in\mathbb{N}_0^3$, we obtain
\begin{align}
\label{EQ140}
    \|\partial^\alpha f\|_{X_{\widetilde\tau}}
    \lec_{c_0,\tilde\tau,\tau,\alpha}
    \|f\|_{X_\tau}.
\end{align}
\end{lemma}

\begin{proof}[Proof of Lemma~\ref{AL01}]
For each $j\in\{1,2,3\}$, the coordinate derivative $\partial_jf$ satisfies
\begin{align}
    \begin{split}
        \llabel{EQ141}
        \|\partial_j f\|_{X_{\tilde{\tau}}}
        &=\sum_{m=0}^\infty\sum_{|\alpha|=m}
        \tilde\tau^{(m-3)_+}\epsilon^\alpha M_m\|\partial^\alpha
        \partial_jf\|
        \\&\lec
        \sum_{m=1}^\infty \sum_{|\alpha|=m}
        \tilde\tau^{(m-4)_+}\epsilon^\alpha M_{m-1}\|\partial^\alpha f\|
        \\&=
        \sum_{m=1}^3\sum_{|\alpha|=m}\epsilon^\alpha M_{m}\|\partial^\alpha f\|
        \underbrace{\frac{M_{m-1}}{M_m}\left(\frac{\tilde\tau}{\tau}\right)^{m-4}}_{\lec 1}
        +\sum_{|\alpha|=4}\tau\epsilon^\alpha M_{m}
        \|\partial^\alpha f\|\underbrace{\tau^{-1}\frac{M_{m-1}}{M_m}
        \left(\frac{\tilde\tau}{\tau}\right)^{m-4}}_{\lec 1}
        \\&\indeq
        +\sum_{m=5}^\infty\sum_{|\alpha|=m}\tau^{(m-3)_+}\epsilon^\alpha M_{m}
        \|\partial^\alpha f\|
        \underbrace{\frac{M_{m-1}}{M_m}
        \left(\frac{\tilde\tau}{\tau}\right)^{m-4}\tau^{-1}}_{\lec 1}
        \\&
        \lec\|f\|_{X_\tau},
    \end{split}
\end{align}
which proves~\eqref{EQ139}. Applying~\eqref{EQ139} repeatedly to $f$, we obtain~\eqref{EQ140}.
\end{proof}
We next record the basic product estimate in analytic norms.

\begin{lemma}[Product estimate in analytic norms]
\label{AL02}
For any $f,g\in X_\tau$, we have
    \begin{align}
    \label{EQ143}
    \|fg\|_{X_\tau}\lec\|f\|_{X_\tau}\|g\|_{X_\tau}.
    \end{align}
Furthermore, if $g\in Y_\tau$, then we have
    \begin{align}
    \llabel{EQ142}
        \|f\nabla g\|_{X_\tau}\lec\|f\|_{X_\tau}(\|g\|_{X_\tau}+\|g\|_{Y_\tau}).
    \end{align}
\end{lemma}
As an application of the product estimate, we prove a quotient estimate in analytic norms.
The key point is that,
under a suitable smallness assumption, the reciprocal can be represented by a
convergent geometric series in~$X_\tau$.

\begin{lemma}[Reciprocal and quotient estimates in analytic norms]
    \label{AL03}
    Let $g\in X_\tau$. Suppose that there exist constants $r\neq 0$
and $c_0\in(0,1/C)$,
where $C$ is the constant provided in~\eqref{EQ143}, such that $\|g-r\|_{X_\tau}<c_0|r|<|r|$ and $g>0$. Then we obtain
    \begin{align}
        \left\|\frac{1}{g}\right\|_{X_\tau}\lec |r|^{-1}
     .
        \label{EQ145}
    \end{align}
Additionally, we have 
    \begin{align}
        \left\|\frac{f}{g}\right\|_{X_\tau}\lec |r|^{-1}{\|f\|_{X_\tau}}
       .
        \label{EQ144}
    \end{align}
\end{lemma}

\begin{proof}[Proof of Lemma~\ref{AL03}]
    It suffices to prove the estimate~\eqref{EQ145} since \eqref{EQ144} follows immediately from \eqref{EQ145} and Lemma~\ref{AL02}.
    Set
    \begin{equation}
    \tilde g=\frac{g-r}{r}.
   \llabel{EQ146}
     \end{equation}
By assumption,
    \begin{equation}
    \|\tilde g\|_{X_\tau}<c_0<1.
   \llabel{EQ147}
     \end{equation}
Since $Cc_0<1$, the Neumann series converges in $X_\tau$ and gives
    \begin{equation}
    \frac{1}{g}=\frac{1}{r(1+\tilde g)}=\frac{1}{r}\sum_{k=0}^\infty (-\tilde{g})^{k}
    .
   \llabel{EQ148}
     \end{equation}
Using Lemma~\ref{AL02}, we have
    \begin{equation}
    \left\|\frac{1}{g}\right\|_{X_\tau}\leq \frac{1}{|r|}\sum_{k=0}^\infty C^k\|\tilde g\|^k_{X_\tau}\lec
    \frac{1}{|r|}
    \sum_{k=0}^\infty(Cc_0)^k\lec_{c_0} \frac{1}{|r|},
   \llabel{EQ149}
     \end{equation}
which concludes the proof. 
\end{proof}

\section*{Acknowledgments}
\rm
IK and QX were supported in part by the NSF grant DMS-2205493.

\colb

\colb

\ifnum\sketches=1
\newpage
\begin{center}
  \bf   Notes?\rm
\end{center}
\huge
\colb

\fi

\end{document}